\documentclass[12pt,reqno]{amsart}
\usepackage{amsmath,amssymb,graphicx,mathrsfs}
\usepackage{stmaryrd,amsthm,color}
\usepackage{pgf}
\usepackage{tikz}
\tikzset{font={\fontsize{10pt}{12}\selectfont}}
\usepackage{pdflscape}
\usetikzlibrary{arrows, decorations.pathmorphing, backgrounds, positioning, fit, petri, automata}
\usepackage{times}
\usepackage[colorlinks=true,allcolors=blue]{hyperref}

\usepackage{amsmath,bm}
\usepackage{geometry}
\numberwithin{equation}{section}
\numberwithin{equation}{section}
\newtheorem{thm}{Theorem}[section]
\newtheorem{prop}[thm]{Proposition}
\newtheorem{lem}[thm]{Lemma}

\theoremstyle{definition}
\newtheorem{rem}[thm]{Remark}
\newtheorem{eg}[thm]{Example}

\renewcommand{\theequation}{\thesection.\arabic{equation}}

\newcommand{\beq}{\begin{equation}}
\newcommand{\eeq}{\end{equation}}

\newcommand{\C}{{\mathbb C}}
\newcommand{\Z}{{\mathbb Z}}

\newcommand{\R}{{\mathbb R}}

\newcommand{\bLa}{{\boldsymbol \La}}
\newcommand{\Gr}{{\mathrm{Gr}}}
\newcommand{\Wr}{{\mathrm{Wr}}}

\newcommand{\cH}{\mathcal{H}}

\newcommand{\cM}{\mathcal{M}}
\newcommand{\cP}{\mathcal{P}}

\newcommand{\bP}{\mathbb{P}}
\newcommand{\cla}{\check{\lambda}}
\newcommand{\cmu}{\check{\mu}}

\newcommand{\sGr}{\mathrm{sGr}}

\newcommand{\bs}{\boldsymbol}
\newcommand{\g}{\mathfrak{g}}
\newcommand{\h}{\mathfrak{h}}
\newcommand{\n}{\mathfrak{n}}
\newcommand{\mc}{\mathcal}
\newcommand{\gl}{\mathfrak{gl}}
\newcommand{\sln}{\mathfrak{sl}}

\newcommand{\la}{\lambda}
\newcommand{\La}{\Lambda}
\newcommand{\sing}{\mathrm{sing}}

\newcommand{\End}{\mathop{\rm End}}

\newcommand{\Ker}{\mathop{\rm Ker}}

\newcommand{\tr}{{\rm tr}}

\newcommand{\gge}{\geqslant}
\newcommand{\lle}{\leqslant}
\newcommand{\opg}{\mathrm{op}_{^t\g}}
\newcommand{\Opg}{\mathrm{Op}_{^t\g}}

\newcommand{\cB}{\mathcal{B}}
\newcommand{\tl}{\tilde}
\newcommand{\mb}{\mathbb}
\newcommand{\pa}{{\partial}}
\newcommand{\calpha}{\check{\alpha}}
\newcommand{\crho}{\check{\rho}}
\makeatletter
\def\@eqnnum{{\normalfont \color{red} (\theequation)}}
\makeatother

\begin{document}
	\pagestyle{myheadings}
	
	\setcounter{page}{1}

	\title{On the Gaudin model of type G$_2$}
	
	\author{Kang Lu and E. Mukhin}
	\address{K.L.: Department of Mathematical Sciences,
		Indiana University-Purdue University\newline
		\strut\kern\parindent Indianapolis, 402 N.Blackford St., LD 270,
		Indianapolis, IN 46202, USA}\email{lukang@iupui.edu}
	\address{E.M.: Department of Mathematical Sciences,
		Indiana University-Purdue University\newline
		\strut\kern\parindent Indianapolis, 402 N.Blackford St., LD 270,
		Indianapolis, IN 46202, USA}\email{emukhin@iupui.edu}

	\begin{abstract}
    We derive a number of results related to the Gaudin model associated to the simple Lie algebra of type G$_2$. 
  
    We compute explicit formulas for solutions of the Bethe ansatz equations associated to the tensor product of an arbitrary finite-dimensional irreducible module and the vector representation. We use this result to show that the Bethe ansatz is complete in any tensor product where all but one factor are vector representations and the evaluation parameters are generic. 
		
    We show that the points of the spectrum of the Gaudin model in type G$_2$ are in a bijective correspondence with self-self-dual spaces of polynomials.
	We study the set of all self-self-dual spaces -- the self-self-dual Grassmannian. We establish a stratification of the self-self-dual Grassmannian with the strata labeled by unordered sets of dominant integral weights and unordered sets of nonnegative integers, satisfying certain explicit conditions. We describe closures of the strata in terms of representation theory.

\medskip

	\end{abstract}
	
	\maketitle
	\thispagestyle{empty}
	\section{Introduction}
	In the last several decades the Gaudin models were a subject of numerous studies. While some papers deal with general situation, see for example \cite{FFRe,F2,R,V1,V}, the most complete results are obtained in type A. One reason is that the underlying geometry is given by osculating Schubert calculus which is well understood, see \cite{MTV1,MTV2}. The types B and C are the next best studied cases, and it is known that the cases of types D, E, and F are the most difficult ones where new methods are needed. 
	
	In this paper we study the Gaudin model in type G$_2$ where one can use the approach similar to types A, B, C. We get help from 
	recent papers  \cite{MRR}, where the explicit formulas for higher Gaudin Hamiltonians of type G$_2$ are described, and \cite{R}, where eigenvalues of the algebra of higher Gaudin Hamiltonians (we call it Bethe algebra) are identified with opers. 
	Our work also builds on ideas of \cite{BM}, where the self-self-dual spaces are defined and studied, and  of \cite{LMV2}, which describes the stratification of the Grassmannians in types A, B, C.
	
    Let us describe our results in more detail. Let $\g$ be the Lie algebra of type  G$_2$. First, we use \cite{MRR} to describe explicitly the Bethe algebra $\mc B$ -- the remarkable commutative subalgebra of the enveloping algebra ${\mc U}(\g[t])$ of the current algebra of $\g$, see Section \ref{sec:Bethe alg}. The Bethe algebra $\mc B$ commutes with $\g$. The objective of the Gaudin model is to study the common eigenvectors and eigenvalues of $\mc B$ acting on $\g[t]$-modules.	
	
    We define the universal differential operator $\mc D^{\mc B}$, the monic differential operator of order $7$ whose coefficients are power series with coefficients in $\mc B$, see Section \ref{sec: diffop}. If $v$ is an eigenvector of $\mc B$  in some $\g[t]$-module $V$ then replacing the coefficients of the universal differential operator with the corresponding eigenvalues we obtain a scalar differential operator $\mc D_v$ with rational coefficients. One of our main results is that if $V$ is a finite-dimensional irreducible $\g[t]$-module then there exists a unique monic polynomial $f$ depending on $V$ such that $f\cdot(\Ker\mc D_v)$ is a self-self-dual space of polynomials with singularities and exponents explicitly determined by $V$. Moreover this determines a bijection between joint eigenvalues of $\mc B$ on finite-dimensional irreducible $\g[t]$-modules and self-self-dual spaces of polynomials, see Theorem \ref{bi rep ssgr}.
	
    The self-self-dual spaces of polynomials are 7-dimensional spaces of polynomials with a remarkable set of symmetries, see Section \ref{sec ssd gr}. These spaces were introduced and studied in \cite{BM}. We call the set of all self-self-dual spaces of polynomials the self-self-dual Grassmannian and denote it by $\mb S\Gr_d$ where $d-1$ is the maximal allowed degree of the polynomials. The set $\mb S\Gr_d$ is an algebraic set of dimension $d-7$. Using the bijection above we are able to describe a stratification of $\mb S\Gr_d$ where the strata and the closures of the strata are described by the representation theory of $\g$, see Theorems \ref{thm sgr dec}, \ref{thm G strata}. Each stratum is a ramified covering of $(\C\mb P_1)^n$ without diagonals with an appropriate $n$, quotient by the free action of an appropriate symmetric group, see Proposition \ref{prop bij BC}. 
   
    The technical part of the paper is the Bethe ansatz method for finding eigenvectors of $\mc B$. While in general this method does not provide the full set of eigenvectors, we show that it does in the case when $V=\bigotimes_{s=1}^n V_{\omega_2}(z_s)$ for generic values of $z_1,z_2,\dots,z_n$, see Theorem \ref{thm G completeness}. Here $V_{\omega_2}$ is the 7-dimensional irreducible $\g$-module and $V_{\omega_2}(z)$ is the corresponding evaluation $\g[t]$-module with evaluation parameter $z\in\C$. We reduce this statement to  the case $V_{\la}(z_1)\otimes V_{\omega_2}(z_2)$, as in \cite{MV1}. In this case we compute the solutions of the Bethe ansatz equations explicitly. The technique for this computation was developed in \cite{LMV1} and the answer is described in Appendix \ref{sec:solutions}. It turns out that in conjunction with \cite{R} this is sufficient to establish the bijection between the $\mc B$-eigenvalues and self-self-dual spaces of polynomials. 
		
    It is expected that the spectrum of the Gaudin model is simple, meaning that each $\mc B$-eigenvalue corresponds to a unique (up to $\g$ action) $\mc B$-eigenvector. Then our bijection can be interpreted as a resolution of singularities in the Bethe ansatz method.
   
    The paper is constructed as follows. We set up our notation for the Lie algebra of type G$_2$ in Section \ref{sec start} and proceed to define the Bethe subalgebra $\mc B$ in Section \ref{sec Gaudin}. We obtain our results on the Bethe ansatz in Section \ref{sec n=2}. The self-self-dual spaces of polynomials and the stratification of $\mb S\Gr_d$ are studied in Section \ref{sec gr}. Section \ref{sec oper} establishes the connection of self-self-dual spaces to the opers which is used in the proofs. We give explicit formulas for solutions of the Bethe ansatz equation in Appendix \ref{sec:solutions} and an example of our stratification of $\mb S\Gr_{11}$ in Appendix \ref{app ex}.
	
	\medskip
	
	{\bf Acknowledgments.} 
	This work was partially supported by a grant from the Simons Foundation  \#353831.
    KL was partially supported by Zhejiang Province Science Foundation, grant No. LY14A010018.
 
	\section{The Lie algebra of type G$_2$}\label{sec start}
	For any Lie algebra $\g$, we denote by $\mc U(\g)$ the universal enveloping algebra of $\g$.
	\subsection{Simple Lie algebra of type G$_2$}\label{sec sla}
	Let $\g$ be the simple Lie algebra of type G$_2$ over $\C$ with the Cartan matrix $A=(a_{i,j})_{i,j=1}^2$,
	\[
	A=\begin{pmatrix}
	2 & -1\\
	-3 & 2
	\end{pmatrix}.
	\]
	
	Note that $\dim \g=14$. Let $D=\mathrm{diag}(d_1,d_2)=\mathrm{diag}(3,1)$, then $B=DA$ is symmetric.
	
	Let $\h\subset\g$ be the Cartan subalgebra and let $\g=\n_-\oplus\h\oplus\n_+$ be the Cartan decomposition. Let $\alpha_1\in\h^*$ be the long root and $\alpha_2\in\h^*$ the short root. Let $\calpha_1,\calpha_2\in \h$ be the corresponding coroots. The set of positive roots is
	\[\alpha_1,\quad \alpha_2,\quad \alpha_1+\alpha_2, \quad \alpha_1+2\alpha_2, \quad \alpha_1+3\alpha_2,\quad 2\alpha_1+3\alpha_2.\]
	
	Let $e_1,e_2$, $\calpha_1,\calpha_2$, $f_1,f_2$ be the Chevalley generators of $\g$. 
	
	Fix a nondegenerate invariant bilinear form $(\ ,\ )$ on $\g$ such that $(\calpha_i,\calpha_j)=a_{i,j}/d_j$. Define the corresponding invariant bilinear form on $\h^*$ such that $(\alpha_i,\alpha_j)=d_ia_{i,j}$. Specifically, we have
	\[
	(\alpha_1,\alpha_1)=6,\quad (\alpha_1,\alpha_2)=-3,\quad (\alpha_2,\alpha_2)=2.
	\]
	We have $\langle \lambda,\calpha_i\rangle=2(\lambda,\alpha_i)/(\alpha_i,\alpha_i)$ for $\lambda\in\h^*$. In particular, $\langle\alpha_j,\calpha_i\rangle=a_{i,j}$. Let $\omega_1,\omega_2\in\h^*$ be the fundamental weights, $\langle \omega_j,\calpha_i\rangle=\delta_{i,j}$.
	
	Let $\mathcal P=\{\la\in\mathfrak h^*|\langle \la,\calpha_i\rangle\in\mathbb Z,~i=1,2\}$ and $\mathcal P^+=\{\la\in\mathfrak h^*|\langle \la,\calpha_i\rangle\in\mathbb Z_{\gge 0},~i=1,2\}$ be the weight lattice and the cone of dominant integral weights. 
	
	Let $\rho=3\alpha_1+5\alpha_2\in {\mathfrak h}^*$, then $\langle \rho,\calpha_i\rangle =1$, $i=1,2$.  We have $(\rho,\alpha_i)=(\alpha_i,\alpha_i)/2$.
	
	For $\la\in {\mathfrak h}^*$, let $V_\la$ be the irreducible $\g$-module with highest weight $\la$. We
	denote $\langle \la,\calpha_i\rangle$ by $\lambda_i$ and sometimes write $(\la_1,\la_2)$ for $\la$ and $V_{(\la_1,\la_2)}$ for $V_\la$.
	
	The Weyl group $ W\subset\mathrm{Aut}(\mathfrak h^*)$ is generated by the simple reflections $s_1,s_2$,
	\[s_i(\la)=\la-\langle \la,\calpha_i\rangle \alpha_i,\quad \la\in\mathfrak h^*.\] The Weyl group of type G$_2$ is isomorphic to the dihedral group D$_6$, in particular, $|W|=12$. We use the notation
	\[\sigma\cdot\la=\sigma(\la+\rho)-\rho,\quad \sigma\in W,\quad\la\in\mathfrak h^*,\]
	for the shifted action of the Weyl group.
	
	The coproduct $\Delta:\mc U(\g)\to \mc U(\g)\otimes \mc U(\g)$ is defined to be the homomorphism of algebras such that $\Delta x=1\otimes x+x\otimes 1$, for all $x\in \g$. Let $(x_i)_{i=1}^{14}$ be an orthonormal basis with respect to the bilinear form $(\ ,\ )$ on $\g$. Let $\Omega=\sum_{i=1}^{14}x_i\otimes x_i\in \g\otimes \g\subset \mc U(\g)\otimes\mc U(\g)$ be the Casimir operator. The operator $\Omega$ is independent of the choice of the orthonormal basis $(x_i)_{i=1}^{14}$.
	
	Let $M$ be a $\g$-module.
	Let $(M)^\sing=\{v\in M~|~\n_+v=0\}$ be the subspace of singular vectors in $M$.
	For $\mu\in\h^*$ let $(M)_{\mu}=\{v\in M~|~hv=\langle \mu,h \rangle v,~\text{for all }h\in\h\}$  be the subspace of $M$ of vectors of weight $\mu$.
	Let $(M)^{\sing}_{\mu}=(M)^\sing\cap (M)_\mu$ be the subspace of singular vectors in $M$ of weight $\mu$.
	
	Given a $\g$-module $M$, denote by $(M)^{\g}$ the subspace of $\g$-invariants in $M$. The subspace $(M)^{\g}$ is the multiplicity space of the trivial $\g$-module in $M$. It is well known that $\dim(V_{\la}\otimes V_{\mu})^{\g}=\delta_{\la,\mu}$ for any $\g$-weights $\la$ and $\mu$.
	
	\subsection{The matrix $G$}\label{sec matrix}
	We define a $7\times7$ matrix $G$ whose entries are in $\g$ in a standard way. 
	Consider $V_{\omega_2}$, the second fundamental representation of $\g$. We have $\dim V_{\omega_2}=7$. Let
	\[
	\pi:\g\to \End(V_{\omega_2}).
	\]This representation is faithful, that is it defines an embedding $\pi: \g\to \gl_7$. 
	
	The nondegenerate invariant bilinear form $(\ ,\ )$ on $\g$ is explicitly given by 
	\[
	(x_1,x_2)=\frac{1}{6}\,\tr\,(\pi(x_1)\pi(x_2)),\qquad x_1,x_2\in \g.
	\]
	
	We choose a basis $\{v_i\}_{i=1}^7$ of $V_{\omega_2}$ such that $v_1$ is a highest weight vector, $v_2=f_2v_1$, $v_3=f_1v_2$, $\sqrt{2} v_4=f_2v_3$, $\sqrt{2} v_5= f_2v_4$, 	$v_6=f_1v_5$, $v_7=f_2v_6$. That allows us to regard elements of $\g$ as $7\times7$ matrices. 
	
	Let $e_{ij}\in \End(\C^7)$ denote the standard matrix units with respect to the basis $\{v_i\}_{i=1}^7$. For example, we have $f_1=e_{32}+e_{65}$, $e_2=e_{12}+\sqrt{2}e_{34}+\sqrt{2}e_{45}+e_{67}$.

    Introduce the form on $\C^7$ by $(v_i,v_j)=(-1)^{i+1}\delta_{i,8-j}$. Set $F_{ij}=e_{ij}-(-1)^{i+j} e_{8-j,8-i}$, $1\leqslant i,j\leqslant 7$. The special orthogonal algebra $\mathfrak{so}_7$ corresponding to the form $(\ ,\ )$ is spanned by $F_{ij}$.
    Clearly, we have $\pi(\g)\subset {\mathfrak{so}_7}$.

	Let
	\[
	G=(\pi\otimes {\rm id})(\Omega)=\sum_{i=1}^{14}\pi(x_i)\otimes x_i=\sum_{i,j=1}^7 e_{ji}\otimes G_{ij} \in \End(\C^7)\otimes \mc U(\g).
	\]
	 
	We use the embedding $\pi$ to describe $G_{ij}$ explicitly. 
	For
	$i,j\in\{2, 3,7\}$ with $i\ne j$ we have $G_{ij}=3 F_{ij}$. Next,
	\[
	G_{22}=2 F_{22}-F_{33}-F_{77},\qquad
	G_{33}=2 F_{33}-F_{22}-F_{77},\qquad
	G_{77}=2 F_{77}-F_{22}-F_{33}.
	\]
	Furthermore,
	$$
	G_{24}=2 F_{24}-\sqrt2 F_{13},\qquad
	G_{34}=2 F_{34}+\sqrt2 F_{67},\qquad
	G_{74}=2 F_{74}-\sqrt2 F_{52},
	$$
	$$
	G_{42}=2 F_{42}-\sqrt2 F_{31},\qquad
	G_{43}=2 F_{43}+\sqrt2 F_{76},\qquad
	G_{47}=2 F_{47}-\sqrt2 F_{25},
	$$
	and
	$$
	G_{24}=-\sqrt2 G_{13},\qquad
	G_{34}=\sqrt2 G_{67},\qquad
	G_{74}=-\sqrt2 G_{52},
	$$
	$$
	G_{42}=-\sqrt2 G_{31},\qquad
	G_{43}=\sqrt2 G_{76},\qquad
	G_{47}=-\sqrt2 G_{25}.
	$$
	The remaining entries of the matrix $G$ are determined by $G_{ij}+(-1)^{i+j}G_{8-j,8-i}=0$.
	
	Note that our conventions are different from \cite{MRR} since we use a different basis for $V_{\omega_2}$, our matrix $G$ and the corresponding matrix in \cite{MRR} are conjugates by an explicit transition matrix.
	
	The elements $G_{22}, G_{33}$ and $G_{ij}$, $i,j\in\{2,3, 4,7\},$ where $i\ne j$, form a basis of $\g$. 
	
	Finally, the Casimir operator $\Omega$ is computed via matrix $G$ as follows.
  Let $G_1,G_2\in \End(\C^7)\otimes \mc U(\g)\otimes \mc U(\g)$ be defined by
	$G_1=\sum_{i,j=1}^7e_{ij}\otimes G_{ji}\otimes 1$ and  $G_2=\sum_{i,j=1}^7e_{ij}\otimes 1\otimes G_{ji}$.
	Then 
	\beq\label{eq:omega}
	\Omega=\dfrac{1}{6}\,\tr\, (G_1G_2).
	\eeq
	
	\subsection{Current algebra $\g[t]$}Let $\g[t] = \g\otimes\C[t]$ be the Lie algebra of $\g$-valued polynomials with the pointwise commutator. We call it the \emph{current algebra} of $\g$. We identify the Lie algebra $\g$ with the subalgebra $\g\otimes 1$ of constant polynomials in $\g[t]$. Hence, any $\g[t]$-module has the canonical structure of a $\g$-module.
	
	It is convenient to collect elements of $\g[t]$ in generating series of a formal variable $x$. For $g\in \g$, set
	\beq\label{eq generating series}
	g(x)=\sum_{s=0}^\infty (g\otimes t^s)x^{-s-1}.
	\eeq
	
	For each $a\in\C$, there exists an automorphism $\tau_a$ of $\g[t]$, $\tau_a:g(x)\to g(x-a)$. Given a $\g[t]$-module $M$, we denote by $M(a)$ the pull-back of $M$ through the automorphism $\tau_a$. As $\g$-modules, $M$ and $M(a)$ are isomorphic by the identity map.
	
	We have the evaluation homomorphism, $\mathrm{ev}: \g[t]\to \g$, $\mathrm{ev} : g(x)\to g x^{-1}$. Its restriction to the subalgebra $\g\subset \g[t]$ is the identity map. For any $\g$-module $M$, we denote by the same letter the $\g[t]$-module, obtained by pulling $M$ back through
	the evaluation homomorphism. For each $a\in\C$, the $\g[t]$-module $M(a)$ is called an \emph{evaluation module}.
	It is well known that all finite-dimensional irreducible $\g[t]$-modules are tensor products of evaluation modules $V_{\la^{(1)}}(z_1)\otimes\dots\otimes V_{\la^{(n)}}(z_n)$ with  dominant integral $\g$-weights $\la^{(1)},\dots,\la^{(n)}$ and distinct evaluation parameters $z_1,\dots,z_n$.
	
	\section{Gaudin Model}\label{sec Gaudin}

    \subsection{Bethe algebra}\label{sec:Bethe alg}The Bethe algebra is a remarkable commutative subalgebra of the universal enveloping algebra $\mc U(\g[t])$ of the current Lie algebra of $\g$. The Bethe algebra of a simple Lie algebra was described in \cite{FFRe} (called the algebra of higher Gaudin Hamiltonians there). An explicit set of generators of the Bethe algebra in the Lie algebra of type G$_2$ is given in \cite{MRR}. We recall this result.

    Denote by $G(x)$ the matrix
    \[
    G(x)=\sum_{i,j=1}^7 e_{ij}\otimes G_{ij}(x),
    \]where $G_{ij}$ are given in Section \ref{sec matrix}, see also \eqref{eq generating series}. 
	Note that $\lim_{x\to \infty} xG(x)=G^{\mathrm{T}}$, where $G^\mathrm{T}$ denotes the transpose of the matrix $G$.
	
	 Denote by $S_2(x)$ and $S_6(x)$ the power series in $x^{-1}$ with coefficients in $\mc U(\g[t])$ obtained by applying an anti-homomorphism which sends $G[-r-1]\mapsto\pa_x^r G(x)/r!$ to (3.6) and (3.7) of \cite{MRR} respectively. Explicitly, 
	\[
	S_2(x)=\tr\, (G(x)^2),
	\]
	\begin{align*}
	S_6(x)=&\,\tr\, (G(x)^6)+5\,\tr\,(G(x)^4G'(x))+7\,\tr\,(G(x)^3G''(x))+114\,\tr\,((G''(x))^2)\\&-639\,\tr\,((G'(x))^3)+31\,\tr\,(G(x)^2(G'(x))^2)-156\,\tr\,(G(x)G'(x)G''(x)).
	\end{align*}
	
	Write
	$$
	S_2(x)=\sum_{i=2}^\infty S_{2\,i}\,x^{-i}\qquad\text{and} \qquad S_6(x)=\sum_{j=6}^\infty S_{6\,j}\,x^{-j},
	$$
	where $S_{2\,i},S_{6\,j}\in\mc U(\g[t])$, $i\in\Z_{\gge 2}$, $j\in\Z_{\gge 6}$. We call the unital subalgebra of $\mc U(\g[t])$ generated by $S_{ij}\in\mc U(\g[t]) $, $i=2,6$, $j\in \Z_{\gge i}$, the \emph{Bethe algebra}  of $\g$ and denote it by $\mc B$.
	
	The Bethe algebra $\cB$ of $\g$ is a commutative subalgebra of $\mc U(\g[t])$ which commutes with the subalgebra $\mc U(\g)\subset \mc U(\g[t])$. As a subalgebra of $\mc U(\g[t])$, the algebra $\cB$ acts on any $\g[t]$-module $M$. Since $\cB$ commutes with $\mc U(\g)$, it preserves the subspace of singular vectors $(M)^\sing$ as well as weight subspaces of $M$. Therefore, the subspace $(M)^{\sing}_{\la}$ is $\cB$-invariant for any
	weight $\la$.
	
	\medskip

    We denote $M(\infty)$ the $\g$-module $M$ with the trivial action of the Bethe algebra $\mc B$. More generally, for a $\g[t]$-module $M'$, we denote by $M'\otimes M(\infty)$ the $\g$-module where we define the action of $\mc B$ so that it acts trivially on $M(\infty)$. Namely, the element $b\in\mc B$ acts on $M'\otimes M(\infty)$ by $b\otimes 1$.
		
	Note that for $g\in\g$, $a\in \C$, and $\g$-module $M$, the action of $g(x)$ on $M(a)$ is given by $g/(x-a)$ on $M$. Therefore, the action of series $S_i(x)$, $i=2,6$, on the module  $M'\otimes M(\infty)$ is the limit of the action of the series $S_i(x)$ on the module $M'\otimes M(z)$ as $z\to \infty$ in the sense of rational functions of $x$. However, such a limit of the action of the coefficients $S_{ij}$, $i=2,6$, $j\in \Z_{\gge i}$, on the module $M'\otimes M(z)$ as $z\to \infty$ does not exist.
	
	Let $M=V_\la$ be an irreducible $\g$-module and let $M'$ be an irreducible finite-dimensional $\g[t]$-module.
		
	\begin{lem}\label{nonzero weight}
	We have an isomorphism of vector spaces: 
	$$\pi:\ (M'\otimes V_{\la})^{\g}\to (M')^{\sing}_{\la}, $$ 
	given by the projection to a lowest weight vector in $V_{\la}$.  
	The map $\pi$ is an isomorphism of $\mc B$-modules $(M'\otimes V_{\la}(\infty))^{\g}\to (M')^{\sing}_{\la}$. \qed
	\end{lem}	
	
	Consider $\bP^1:=\C\cup\{\infty\}$. Set
	\[
	{\mathring{\bP}}_n:=\{\bm z=(z_1,\dots,z_n)\in (\mb{P}^1)^n~|~z_i\ne z_j\ \text{ for }\ 1\lle i<j\lle n\},
	\]
	\[
	\R{\mathring{\bP}}_n:=\{\bm z=(z_1,\dots,z_n)\in {\mathring{\bP}}_n~|~z_i\in\R \text{ or }z_i=\infty,\ \text{ for }\ 1\lle i\lle n\}.
	\]
	
	We are interested in the action of the Bethe algebra $\cB$ on the tensor product $\bigotimes_{s=1}^n V_{ \la^{(s)}}(z_s)$, where
	$\bLa=(\la^{(1)},\dots,\la^{(n)})$ is a sequence of dominant integral $\g$-weights and $\bm z=(z_1,\dots,z_n)\in {\mathring{\bP}}_n$. By Lemma \ref{nonzero weight}, it is sufficient to consider the spaces of invariants $(\bigotimes_{s=1}^n V_{\la^{(s)}}(z_s))^{\g}$. For brevity, we write $V_{\bLa,\bm z}$ for the $\mc B$-module $\bigotimes_{s=1}^n V_{ \la^{(s)}}(z_s)$ and $V_{\bLa}$ for the $\g$-module $\bigotimes_{s=1}^n V_{\la^{(s)}}$. 
	
	In what follows we also use the notation $V_{\bLa}$ and $V_{\bLa,\bm z}$ when $\bLa=(\la^{(1)},\dots,\la^{(n)})$ is an arbitrary sequence of $\g$-weights and $\bm z\in (\bP^1)^n$.

	\subsection{Universal differential operator}\label{sec: diffop}
	Define the power series $B_i(x)$, $i=2,3,\dots,7$, in $x^{-1}$ with coefficients in $\mc U(\g[t])$ by the formula:
	\begin{align*}
		&B_2(x)=-\frac{1}{6}S_2(x), \qquad\qquad\qquad\qquad\,\,\,\, \,\,\, B_3(x)=-\frac{5}{12}S_2'(x),\\
	&B_4(x)=\frac{1}{144}S_2^2(x)-\frac{1}{2}S_2''(x),\qquad\quad\,\,\,\,\,\,\,\,\, B_5(x)= \frac{1}{48}S_2(x)S_2'(x)-\frac{1}{3}S_2'''(x),\\
	&B_6(x)=\frac{1}{162}\Big(S_6(x)-\frac{11}{144}S_2^3(x)-\frac{7}{8}\big(S_2'(x)\big)^2-\frac{5}{2}S_2(x)S_2''(x)-96S_2^{(4)}(x)\Big),\\
	&B_7(x)=\frac{1}{15552}\Big(48S_6'(x)-3960S_2^{(5)}(x)-174S_2(x)S_2'''(x)\\ &\quad\qquad\qquad\qquad\qquad\qquad\qquad\qquad\qquad \,-366S_2'(x)S_2''(x)-11S_2^2(x)S_2'(x)\Big).
	\end{align*}
	
	The following lemma is clear. 
	\begin{lem}\label{B2 and B6}
	Coefficients of $B_i(x)$ are in the Bethe algebra $\mc B$. Coefficients of $B_2(x)$ and $B_6(x)$ generate the Bethe algebra $\mc B$. \qed
	\end{lem}
	\begin{rem}\label{rem B}
		The  power series $B_i(x)$, $i=2,3,\dots,7$, are computed from the preimages of the elements $w_i$, $i=2,3,\dots,7$, in classical $\mc W$-algebra under the affine Harish-Chandra isomorphism, see Section 4 of \cite{MRR}. Note that the following relation among $w_i$ in \cite{MRR} is true,
		\[
		8 w_7= 4 w_6' - 2 w_2^{(5)} - w_2w_2''' -3w_2'w_2''.
		\]
	\end{rem}
	
	\medskip
	
    Now we define \emph{the universal differential operator} $\mc{D^B}$ by
	\[
	\mc{D^B}=\pa_x^7+\sum_{i=2}^7B_i(x)\pa_x^{7-i}.
	\]
	The operator $\mc{D^B}$ is a monic differential operator of order $7$ whose coefficients are power series in $x^{-1}$ with coefficients in $\mc B\subset {\mc U} (\g[t])$.
	
	Let $M$ be a module over the Bethe algebra $\mc B$ and let $v\in M$ be a common eigenvector of $\mc B$: $B_i(x)v = h_i(x)v$, $i=2,\dots,7$. Then we call the scalar differential operator
	\beq\label{scalar diff op}
	\mathcal D_v=\pa_x^7+\sum_{i=2}^7h_i(x)\pa_x^{7-i}
	\eeq
	the \emph{differential operator associated with the eigenvector $v$.}
	
    Let $v\in V_{\bLa,\bs z}$ be a common eigenvector of $\mc B$, where $\bLa$ is a sequence of dominant integral $\g$-weights.
    Our choice of generators is such that the kernel of $f\cdot\mc D_v\cdot f^{-1}$ is a pure self-self-dual space of polynomials for some polynomial $f$, see Proposition \ref{prop diff oper} and Theorem \ref{bi rep ssgr} below.
    
    \subsection{Quadratic Gaudin Hamiltonians}
    Let $n$ be a positive integer and $M_1,\dots, M_n$ a sequence of $\g$-modules. Set $\bs M=M_1\otimes\dots\otimes M_n$.
	
	Let $\bm z=(z_1,\dots,z_n)$ be a sequence of distinct complex numbers. Introduce linear operators $\cH_1(\bm z),\dots,\cH_n(\bm z)$ in $\End(\bs M)$ by the formula
    $$
    \cH_i(\bm z)=\sum_{j,\ j\ne i}\frac{\Omega^{(i,j)}}{z_i-z_j},\quad i=1,\dots,n,
    $$
	where $\Omega^{(i,j)}$ denotes the Casimir operator acting on the $i$-th and $j$-th factors of $\bs M$. The operators $\cH_1(\bm z),\dots,\cH_n(\bm z)$ are called the \emph{Gaudin Hamiltonians} associated with $\bs M$. 
    
    Denote by $B_2(x;\bs M,\bm z)\in \End(\bs M)$ the corresponding linear operator obtained from the action of $B_2(x)$ on $\bigotimes_{s=1}^n M_s(z_s)$. Using \eqref{eq:omega}, one obtains
    \[
    \mc H_i(\bm z)=-\mathrm{Res}_{x=z_i} B_2(x;\bs M,\bm z),\qquad i=1,\dots,n.
    \]
    In particular, the Gaudin Hamiltonians commute, $[\cH_i(\bm z),\cH_j(\bm z)]=0$ for all $i,j$. Moreover, the Gaudin Hamiltonians commute with the action of $\g$, $[\cH_i(\bm z),g]=0$ for all $i$ and $g\in \g$. Hence for any $\mu\in\h^*$, the Gaudin Hamiltonians preserve the subspace $ (\bs M)_\mu^\sing\subset \bs M$ of singular vectors of weight $\mu$.  
    
	\section{Completeness of Bethe ansatz of  $V_{\la}\otimes V_{\omega_2}^{\otimes k}$}\label{sec n=2}
	\subsection{Bethe ansatz}\label{sec:Bethe ansatz}
	Fix a sequence of weights $\bLa=(\la^{(1)},\dots,\la^{(n)})$, $\la^{(s)}\in\h^*$, and a pair of non-negative integers $\bm{l}=(l_1,l_2)$. Set $l=l_1+l_2$. Denote $\La=\la^{(1)}+\dots+\la^{(n)}$ and $\alpha(\bm l)=l_1\alpha_1+l_2\alpha_2$. Let $\bm z=(z_1,\dots,z_n)$ be a sequence of distinct complex numbers.
	
	The \emph{Bethe ansatz equation associated to $\bm \La, \bm z,\bm l$} is the system of algebraic equations for complex variables $\bm t=(t_1^{(1)},\dots,t_{l_1}^{(1)};t_1^{(2)},\dots,t_{l_2}^{(2)})$:
	\beq\label{eq:bae.a}
	-\sum_{s=1}^n\frac{( \la^{(s)},\alpha_1)}{t_i^{(1)}-z_s}-\sum_{k=1}^{l_2}\frac{3}{t_i^{(1)}-t_k^{(2)}}+\sum_{k=1,k\ne i}^{l_1}\frac{6}{t_i^{(1)}-t_k^{(1)}}=0,\quad i=1,\dots,l_1,
	\eeq
	\beq\label{eq:bae.b}
	-\sum_{s=1}^n\frac{( \la^{(s)},\alpha_2)}{t_i^{(2)}-z_s}-\sum_{k=1}^{l_1}\frac{3}{t_i^{(2)}-t_k^{(1)}}+\sum_{k=1,k\ne i}^{l_2}\frac{2}{t_i^{(2)}-t_k^{(2)}}=0,\quad i=1,\dots,l_2.
	\eeq
	
	By definition, if $\bm t$ is a solution of the Bethe ansatz equation, then all $t_i^{(j)}$ are distinct. Also if $(\la^{(s)},\alpha_j)\ne0$ for some $j,s$, then $t_i^{(j)}\ne z_s$ for $i=1,\dots,l_j$.
	
	Let $\mathfrak S_m$ be the permutation group of the set $\{1,\dots,m\}$. The groups $\mathfrak S_{l_j}$, $j=1,2$, act on the set of solutions of the Bethe ansatz equation associated to $\bm \La, \bm z,\bm l$ by permuting the coordinates with the same upper index $(j)$. Hence $\mathfrak S_{\bm l}:=\mathfrak S_{l_1}\times \mathfrak S_{l_2}$ acts freely on the set of solutions of the Bethe ansatz equation associated to $\bm \La, \bm z,\bm l$. In what follows we do not distinguish between solutions in the same $\mathfrak S_{\bm l}$-orbit.
	
	\medskip
	
	Following \cite{SV}, there exists a well-defined rational map
	\[\psi:\C^n\times\C^{l}\to (V_{\bm\La})_{\La-\alpha(\bm l)}\]\
	called \emph{the canonical weight function}.
	
	Let $\bm t\in\C^l$ be a solution of the Bethe ansatz equation associated to $\bm \La, \bm z,\bm l$. Then the value of the weight function $\psi(\bm z,\bm t)\in (V_{\bm\La})_{\La-\alpha(\bm l)}$ is called the {\it Bethe vector}. The Bethe vector does not depend on a choice of the representative in the $\mathfrak S_{\bm l}$-orbit of solutions.
	
	It is known from Lemma 2.1 of \cite{MV2} that if the weight $\La-\alpha(\bs l)$ is dominant integral, then the set of solutions of the Bethe ansatz equation is finite. 
	
	Assume that $\La-\alpha(\bs l)$ and all $\la^{(s)}$ are dominant integral $\g$-weights.
    The following facts about Bethe vectors are known.  	
	\begin{thm}[\cite{RV}]\label{thm:bveigen}
		The Bethe vector $\psi(\bm z,\bm t)$ is singular,
		$\psi(\bm z,\bm t)\in (V_{\bm\La})_{\La-\alpha(\bm l)}^\sing$. Moreover, $\psi(\bm z,\bm t)$ is a common eigenvector of the Gaudin Hamiltonians.\qed
	\end{thm}
	\begin{thm}[\cite{FFRe}]
		The Bethe vector $\psi(\bm z,\bm t)$ is a common eigenvector of $\mc B$.\qed
	\end{thm}
	\begin{thm}[\cite{V}]\label{non-zero}
The Bethe vector $\psi(\bm z,\bm t)$ is nonzero.\qed
\end{thm}

	\subsection{Reproduction procedure}
	Introduce a pair of polynomials  $\bm y=(y_1(x),y_2(x))$ in a variable $x$ by the formula
	\[
	y_j(x)=\prod_{i=1}^{l_j}(x-t_i^{(j)}),\quad j=1,2.
	\]
	We say that the pair of polynomials $\bm y$ \emph{represents $\bs t$}.
	Note that the pair $\bm y$ does not depend on a choice of the representative in the $\mathfrak S_{\bs l}$-orbit of $\bm t$. Our choice is such that $y_j$ are monic polynomials, however, we often consider $y_j$ up to multiplication by a nonzero number, that is we regard $\bs y$ as an element in $(\bP\C[x])^{2}$.
	
	Introduce functions
	\beq\label{eq poly T}
	\mc T_j(x)=\prod_{s=1}^n(x-z_s)^{\langle\la^{(s)},\calpha_j\rangle},\quad j=1,2.
	\eeq
	
	We say that a given pair of polynomials $\bm y$ is \emph{generic with respect to $\bs \La, \bs z$} if the following conditions are met:
	\begin{enumerate}
		\item the polynomials $y_j(x)$ have no multiple roots;
		\item the roots of $y_j(x)$ are different from the roots and singularities of the function $\mc T_j$; 
		\item the polynomials $y_1(x)$, $y_2(x)$ have no common roots.
	\end{enumerate}
	
	If $\bm y$ represents a solution of the Bethe ansatz equation associated to $\bm \La$, $\bm z$, $\bm l$, then $\bm y$ is generic with respect to $\bm \La, \bm z$.
	
	For $g_1(x),\dots,g_m(x)\in \C[x]$, denote by $\Wr(g_1(x),\dots,g_m(x))$ the \emph{Wronskian},
	\[
	\Wr(g_1(x),\dots,g_m(x))=\det(d^{i-1}g_j/dx^{i-1})_{i,j=1}^m.
	\]
	
	\begin{thm}[\cite{MV4}]\label{thm:reproducation}
		Assume that $\bs z\in\C^n$ has distinct coordinates and $z_1=0$. Assume that $\la^{(i)}\in\mathcal P^+$, $i=2,\dots,n$.
		A generic pair $\bm y=(y_1,y_2)$ represents a solution of the Bethe ansatz equation associated to $\bs \La,\bs z,\bs l$ if and only if there exist polynomials $\tilde y_1,\tl y_2$ satisfying
		\beq\label{3}
		\Wr(y_1,x^{\langle \la^{(1)}+\rho,\calpha_1\rangle} \tilde y_1)=\mc T_1y_2,\quad 	
		\Wr(y_2,x^{\langle \la^{(1)}+\rho,\calpha_2\rangle} \tilde y_2)=\mc T_2y_1^3.
		\eeq
		Let $\bm y^{(1)}=(\tl y_1,y_2)$ and $\bm y^{(2)}=( y_1,\tl y_2)$. If for some $i\in\{1,2\}$, the pair $\bm y^{(i)}$ is generic, then it represents a solution of the Bethe ansatz equation associated to data $(s_i\cdot\la^{(1)},\la^{(2)},\dots,\la^{(n)}), \bs z, \bm{l_i}$, where $\bm{l_i}$ is determined by the equation $$\La-\la^{(1)}-\alpha(\bm{l_i})=s_i(\La-\la^{(1)}-\alpha(\bs l)).$$
		\hfill$\square$
	\end{thm}
	We say that the pair $\bm y^{(i)}$
	is obtained from the pair $\bm  y$ by {\it the reproduction procedure in the $i$-th direction}.
	
	Note that the reproduction procedure can be iterated. We use the notation $\bm y^{(i)(j)}$ for $\big(\bm y^{(i)}\big)^{(j)}$. We have $\bs y^{(i)(i)}=\bs y$. More generally, it is shown in \cite{MV4}, that the pairs obtained from $\bs y$ by iterating a reproduction procedure are in a bijective correspondence with the elements of the Weyl group of $\g$.
	
	\subsection{Differential operator associated with a Bethe vector} Let $\bLa=(\la^{(1)},\dots,\la^{(n)})$ be a collection of dominant integral $\g$-weights and assume that $\La-\alpha(\bm l)$ is dominant integral. 
	
	Let $\bm y$ represent a solution $\bs t$ of the Bethe ansatz equation associated to $\bm \La, \bm z,\bm l$. Recall the scalar differential operator $\mc D_{\psi(\bm z,\bm t)}$ associated with the corresponding Bethe vector, see \eqref{scalar diff op}. The goal of this section is to describe  $\mc D_{\psi(\bm z,\bm t)}$ in terms of $\bs y$.
	
	Following \cite{BM}, let $\mc D_{\bm y}$ be the monic linear differential operator  given by the formula
	\begin{align}\label{eq dy}
	\mc D_{\bm y}=&\Big(\pa_x-\ln'\Big(\frac{\mc T_1^4\mc T_2^2}{y_1}\Big)\Big)\Big(\pa_x-\ln'\Big(\frac{\mc T_1^3\mc T_2^2y_1}{y_2}\Big)\Big)\Big(\pa_x-\ln'\Big(\frac{\mc T_1^3\mc T_2y_2}{y_1^2}\Big)\Big)\nonumber\\
	& \times\Big(\pa_x-\ln'(\mc T_1^2\mc T_2)\Big)\Big(\pa_x-\ln'\Big(\frac{\mc T_1\mc T_2y_1^2}{y_2}\Big)\Big)\Big(\pa_x-\ln'\Big(\frac{\mc T_1y_2}{y_1}\Big)\Big)\Big(\pa_x-\ln'(y_1)\Big).
	\end{align}
	The kernel of the operator $\mc D_{\bm y}$ is the space spanned by the first coordinates of the G$_2$-population originated at $\bm y$. 
	In particular, the kernel of $\mc D_{\bm y}$ consists of polynomials only. Moreover $\Ker(\mc D_{\bm y})$ is a pure self-self-dual space of polynomials, see the definition of pure self-self-dual space in Section \ref{sec ssd gr} below.
	
	\begin{prop}\label{prop diff oper}
	Let $\bm y$ represent a solution $\bs t$ of the Bethe ansatz equation associated to $\bm \La, \bm z,\bm l$. Then we have $\mc D_{\psi(\bm z,\bm t)}=(\mc T_1^2\mc T_2)^{-1}\cdot\mc D_{\bm y}\cdot (\mc T_1^2\mc T_2)$.
	\end{prop}
	\begin{proof}
		The statement follows from Remark \ref{rem B} and Section 5 of \cite{MRR}, cf. also \cite{FFRe}.
	\end{proof}
	
	We generalize Proposition \ref{prop diff oper} for arbitrary eigenvectors of the Bethe algebra in Theorem \ref{bi rep ssgr} below.
	Now we prove the following lemma.
	
	Let $\bLa=(\la^{(1)},\dots,\la^{(n)})$ be a sequence of dominant integral $\g$-weights and let $\bs z=(z_1,\dots,z_n)$ be a sequence of distinct complex numbers. Let $\la$ be a dominant integral $\g$-weight.
	
	\begin{lem}\label{exponents} 
	Let $v\in (V_{\bLa,\bs z})_\la^{\sing}$ be an eigenvector of the Bethe algebra $\mc B$. Then the operator $\mathcal D_v$ is a monic Fuchsian differential operator with singularities at $z_s$, $s=1,\dots,n$, and $\infty$ only. The exponents of $\mathcal D_v$ at $z_s$ are
	\beq\label{exp at z}
	-2\la_1^{(s)}-\la_2^{(s)},-\la_1^{(s)}-\la_2^{(s)}+1,-\la_1^{(s)}+2,3,\la_1^{(s)}+4,\la_1^{(s)}+\la_2^{(s)}+5,2\la_1^{(s)}+\la_2^{(s)}+6,\eeq 
and the exponents at $\infty$ are 
\beq\label{exp at inf} -2\la_1-\la_2-6,-\la_1-\la_2-5,-\la_1-4,-3,\la_1-2,\la_1+\la_2-1,2\la_1+\la_2.\eeq
	\end{lem}
	\begin{proof}
	Let $v^+$ be the highest weight vector of $V_{\bLa,\bs z}$, then $v^+$ is a Bethe vector and the corresponding operator $\mathcal D_{v^+}$ corresponds to $\bs y=(1,1)$ in Proposition \ref{prop diff oper}. In this case it is easy to see that the singular points and exponents are as in the lemma. 
	
	Let $\mu$ be a $\g$-weight and $z\in\C$. Then $B_i(x)$, $i=2,\dots,7$, act on the evaluation module $V_\mu(z)$ as $c_i(\mu)/(x-z)^i$, where $c_i(\mu)$ are scalars depending on $\mu$. Indeed, it follows from the standard fact that $\tr((G^{\mathrm{T}})^k)$ are central in $\mc U(\g)$ for all $k\in\Z_{\gge 0}$. Therefore the expansion of the operator $B_i(x)$ acting on $V_{\bLa,\bs z}$ as a function of $x$ around $x=z_s$ has the form $c_i(\la^{(s)})/(x-z_s)^i(1+O(x-z_s))$. It follows that the exponents of $\mathcal D_{v^+}$  and $\mathcal D_v$ at $z_s$ are the same.
	
    To compute the exponents at $x=\infty$, observe that the expansion of  $B_i(x)$ 
    acting on $V_{\bLa,\bs z}$ at $x=\infty$ has the form $\tilde c_i/x^i(1+O(1/x))$ where $\tilde c_i$ are operators which do not depend on $\bs z$. Consider $(\bigotimes_{s=1}^n V_{\la^{(s)}}(0))_{\la}^{\sing}$. Then it is clear that $\tilde c_i$ are scalar operators on this space. Moreover, the value of the scalar is given by the action of $B_i(x)$ on the highest weight vector $\tl v^+\in V_\la(0)$. Note that 
    $\tl v^+$ is a Bethe vector corresponding to  $\bs y=(1,1)$ in Proposition \ref{prop diff oper} and the exponents of $\mc D_{\tl v^+}$ and $\mc D_v$ at $\infty$ are the same, it follows that the exponents of $\mc D_v$ at $\infty$ are as in the lemma. 
	\end{proof}
	
	\subsection{Parameterization of solutions}\label{sec2.6}
	In Sections \ref{sec2.6}, \ref{sec strategy}, \ref{sec G generic}, we work with data $\bs \La=(\la,\omega_2)$, $\bs z=(0,1)$, where $\la\in\cP^+$. The main results of these sections
	are the explicit formulas for the solutions of the Bethe ansatz equation associated to $\bLa$, $\bm z$, $\bm l$, see Theorem \ref{thm G2 generic} and Appendix \ref{sec:solutions}, and the completeness of Bethe ansatz for $V_\la\otimes V_{\omega_2}^{\otimes k}$, see Theorem \ref{thm G completeness}.
	
	Let $\la\in\mathcal P^+$ and recall that $\la_i=\langle \la,\calpha_i\rangle$. We write the decomposition of the finite-dimensional $\g$-module $V_\la\otimes V_{\omega_2}$. We have
	\begin{align}\label{eq:dec of B}
	V_{\la}\otimes V_{\omega_2}=&V_{\la+\omega_2}\oplus V_{\la+\omega_2-\alpha_2}\oplus V_{\la+\omega_2-\alpha_1-\alpha_2}\oplus V_{\la+\omega_2-\alpha_1-2\alpha_2} \nonumber\\& \qquad\qquad\qquad\quad\,\,\oplus V_{\la+\omega_2-\alpha_1-3\alpha_2}\oplus V_{\la+\omega_2-2\alpha_1-3\alpha_2}\oplus V_{\la+\omega_2-2\alpha_1-4\alpha_2}\nonumber\\
	=&V_{(\la_1,\la_2+1)}\oplus V_{(\la_1+1,\la_2-1)}\oplus
	V_{(\la_1-1,\la_2+2)}\oplus V_{(\la_1,\la_2)}\nonumber\\ & \qquad\qquad\qquad\qquad\quad\,\,\,\,\oplus V_{(\la_1+1,\la_2-2)} \oplus V_{(\la_1-1,\la_2+1)}\oplus V_{(\la_1,\la_2-1)},
	\end{align}
	with the convention that the summands with non-dominant highest weights are omitted. In addition, if $\la_2=0$, then the summand $V_{\la+\omega_2-\alpha_1-2\alpha_2}=V_{(\la_1,\la_2)}$ is absent.
	
	Note, in particular, that all multiplicities are 1.
	
	By Theorem \ref{thm:bveigen}, to diagonalize the Gaudin Hamiltonians acting on $V_\la\otimes V_{\omega_2}$, it is sufficient to find a solution of the Bethe ansatz equation \eqref{eq:bae.a}-\eqref{eq:bae.b} associated to $\bs\La,\bs z$ and $\bs l$ corresponding to the summands in the decomposition (\ref{eq:dec of B}).

    We call a pair of integers $\bs l$ {\it admissible} if $V_{\la+\omega_2-\alpha(\bs l)}\subset V_\la\otimes V_{\omega_2}$.	
    Any admissible pair has the form
    \beq\label{adm l}\bm l_0=(0,0),\, \bm l_1=(0,1),\, \bm l_2=(1,1),\, \bm l_3=(1,2),\, \bm l_4=(1,3),\, \bm l_5=(2,3),\, \bm l_6=(2,4).\eeq
	
	\subsection{The explicit $2$-point solutions of the Bethe ansatz equation}\label{sec strategy}
	To solve the Bethe ansatz equation with a given admissible $\bm l$, we use the method of Section 3 in \cite{LMV1}. 
	
	All weights in $V_{\omega_2}$ are in the Weyl orbit of $\omega_2$ with the exception of the weight $\mu=0$. Explicitly, the weights in $V_{\omega_2}$ are
	\[
	\omega_2,\ \ s_2\omega_2,\ \ s_1s_2\omega_2,\ \ 0,\ \ s_2s_1s_2\omega_2,\ \ s_1s_2s_1s_2\omega_2,\ \ s_2s_1s_2s_1s_2\omega_2.
	\]
	The pair $\bm y=(1,1)$ represents the unique solution of the Bethe ansatz equation associated to $(\theta,\omega_2),\bm z,\bm l_0$. Using this solution and Theorem \ref{thm:reproducation}, we construct solutions of the Bethe ansatz equation related to $\bm l_i$, $i\ne 3$, as follows. If $\bm y^{(2)}$ is generic, then it represents a solution of the Bethe ansatz equation associated to $(s_2\cdot \theta,\omega_2),\bm z,\bm l_1$. If $\bm y^{(2)(1)}$ is generic, then it represents a solution of the Bethe ansatz equation associated to $(s_1s_2\cdot \theta,\omega_2),\bm z,\bm l_2$. Continue iterating the reproduction procedure and suppose that it is applicable in every step. Finally,  if $\bm y^{(2)(1)(2)(1)(2)}$ is generic, then it represents a solution of the Bethe ansatz equation associated to $(s_2s_1s_2s_1s_2\cdot \theta,\omega_2),\bm z,\bm l_6$. 
	
	To solve the Bethe ansatz equation associated to $(\la,\omega_2),\bm z,\bm l_i$, we take $\theta$ in the following way,
	\[
	\theta=\la,\quad i=0; \qquad \theta=s_2\cdot\la=(\la_1+\la_2+1,-\la_2-2),\quad i=1; 
	\]
	\[
	\theta=s_2s_1\cdot\la=(2\la_1+\la_2+2,-3\la_1-\la_2-5),\quad i=2;
	\]
	\[
	\theta=s_2s_1s_2\cdot\la=(2\la_1+\la_2+2,-3\la_1-2\la_2-6),\quad i=4;\]
	\[ 
	\theta=s_2s_1s_2s_1\cdot\la=(\la_1+\la_2+1,-3\la_1-2\la_2-6),\quad i=5; 
	\]
	\[
	 \theta=s_2s_1s_2s_1s_2\cdot\la=(\la_1,-3\la_1-\la_2-5),\quad i=6.
	\]
	 By considering the Wronskian equalities \eqref{3}, we solve for $\bm y^{(2)}$, $\bm y^{(2)(1)}$, $\bm y^{(2)(1)(2)}$, $\bm y^{(2)(1)(2)(1)}$, and $\bm y^{(2)(1)(2)(1)(2)}$ explicitly in terms of $\theta$. 
	Then we write $\theta$ in terms of $\la$ and check that the resulting pairs of polynomials are generic with respect to $\bLa,\bm z, \bm l_i$ under the condition of admissibility, see Section 3 of \cite{LMV1} for more detail. 
	
	The case of $\bm l_3$ is solved directly.
	
	We obtain the following statement.
	
	\begin{thm}\label{thm G2 generic}
		Let $i\in \{0,1,\dots,6\}$ and $\la\in\cP^+$. Suppose $\bm l_i$ is admissible, then $\bm {y^{l_i}}$ in Appendix \ref{sec:solutions} represents a solution to the Bethe ansatz equation associated to $\bs\La=(\la,\omega_2)$, $\bs z=(0,1)$, and $\bm l_i$.
	\end{thm}
	\begin{proof}
		The statement can be checked directly with the help of computer.
	\end{proof}
	
	\begin{rem}
		In fact, if $\bm l_i$ is admissible, then $\bm {y^{l_i}}$ in Appendix \ref{sec:solutions} represents the unique solution to the Bethe ansatz equation associated to $\bs\La=(\la,\omega_2)$, $\bs z=(0,1)$, and $\bm l_i$. This follows from  Theorem \ref{non-zero} together with Theorem 5.2 in \cite{V1} since the multiplicity spaces are one-dimensional. 
	\end{rem}
	
	\subsection{Completeness of Bethe ansatz}\label{sec G generic}
	The main result of the section is Theorem \ref{thm G completeness}.

	Recall that the value of the weight function $\psi(z_1,z_2,\bm t)$ at a solution of the Bethe ansatz equation is called the Bethe vector.
	We have the following result, which is usually referred to as the completeness of Bethe ansatz.
	\begin{thm}\label{thm3.14}
		The set of Bethe vectors $\psi(z_1,z_2,\bm t)$, where $\bm t$ runs over the solutions to the Bethe ansatz equations with admissible $\bm l$, forms a basis of $(V_\lambda\otimes V_{\omega_2})^\sing$.
	\end{thm}
	
	\begin{proof}
		All multiplicities in the decomposition of $V_\lambda\otimes V_{\omega_2}$ are 1. By Theorem \ref{thm G2 generic} for each admissible $\bm l$ we have a solution of the Bethe ansatz equation. As the Bethe vectors $\psi(z_1,z_2,\bm t)$ are nonzero and singular, the theorem follows.
	\end{proof}
	
	From Theorem \ref{thm3.14} we obtain that the Bethe ansatz is complete in a more general setting.
	\begin{thm}\label{thm G completeness}
		Let $\la\in\cP^+$ and $k\in \Z_{\gge 0}$. For a generic $(k+1)$-tuple of distinct complex numbers $\bm z=(z_0,z_1,\dots,z_k)$, the Gaudin Hamiltonians $\cH_0,\cH_1,\dots,\cH_k$ acting on $(V_\la\otimes V_{\omega_2}^{\otimes k})^\sing$ are diagonalizable and have simple joint spectrum. Moreover, for generic $\bm z$ there exists a set of solutions $\{\bm t_i,~i\in I\}$ of the Bethe ansatz equations such that the corresponding Bethe vectors $\{\psi(\bm z,\bm t_i),~i\in I\}$ form a basis of $(V_\la\otimes V_{\omega_2}^{\otimes k})^\sing$.
	\end{thm}
	\begin{proof}
		It is similar to the proof of Theorem 6.1 in \cite{MV1}, see also Theorem 4.5 in \cite{LMV1}.
	\end{proof}

	\section{Self-self-dual Grassmannian}\label{sec gr}
	Let $N$, $d\in \Z_{>0}$ be such that $N\lle d$. A sequence of integers $\la = (\la_1,\dots,\la_N)$ such that $\la_1\gge\la_2\gge\dots\gge\la_N\gge 0$ is called
	\emph{a partition with at most $N$ parts}. Set $|\la|=\sum_{i=1}^N\la_i$. Then it is said that $\la$ is a partition of $|\la|$.
	\subsection{Schubert cells}\label{sec schubert}
	Let $\C_d[x]$ be the space of polynomials in $x$ with complex coefficients of degree less than $d$. We have $\dim \C_d[x] = d$. Let $\mathrm{Gr}(N, d)$ be the Grassmannian of all $N$-dimensional subspaces in $\C_d[x]$. The Grassmannian $\Gr(N,d)$ is a
	smooth projective complex variety of dimension $N(d-N)$. 
	
	For $X\in\Gr(N,d)$, let $\Wr(X)$ be the unique monic polynomial of the form $\Wr(u_1,\dots,u_N)$, $u_i\in X$, $i=1,\dots,N$.

	Let $\R_d[x]\subset\C_d[x]$ be the set of polynomials in $x$ with real coefficients of degree less than $d$.
	Let $\Gr^\R(N,d)\subset \Gr(N,d)$ be the set of subspaces which have a basis consisting of polynomials with real coefficients. For $X\in\Gr(N,d)$ we have 	$X\in\Gr^\R(N,d)$ if and only if $\dim_\R(X\cap \R_d[x])=N$. We call such points $X$ {\it real}.

	For a full flag $\mc F=\{0\subset F_1\subset F_2\subset \dots\subset F_d=\C_{d}[x]\}$ and a partition $\la=(\la_1,\dots,\la_N)$ such that $\la_1\lle d-N$, the Schubert cell $\Omega_{\la}(\mc F)\subset\Gr(N,d)$ is given by
	\[
	\Omega_{\la}(\mc F)=\{X\in\Gr(N,d)|\dim(X\cap F_{d-j-\la_{N-j}})=N-j,~\dim(X\cap F_{d-j-\la_{N-j}-1})=N-j-1\}.
	\]   We have $\mathrm{codim}~\Omega_{\la}(\mc F)=|\la|$.
	
	The Schubert cell decomposition associated to a full flag $\mc F$ is given by
	\[
	\Gr(N,d)=\bigsqcup_{\la,\ \la_1\lle d-N}\Omega_{\la}(\mc F).
	\]
	The Schubert cycle  $\overline{\Omega}_{\la}(\mc F)$ is the closure of a Schubert cell $\Omega_{\la}(\mc F)$ in the Grassmannian $\Gr(N,d)$. Schubert cycles are algebraic sets with very rich geometry and topology.
	It is well known that Schubert cycle  $\overline{\Omega}_{\la}(\mc F)$  is described by the formula
	\[
	\overline{\Omega}_{\la}(\mc F)=\bigsqcup_{\substack{\la\subseteq\mu,\\ \mu_1\lle d-N}}\Omega_{\mu}(\mc F).
	\]
	Here the union is over partitions $\mu$ with at most $N$ parts such that $\mu_i\gge \la_i$, $i=1,\dots, N$.
	
	Let $\mc F(\infty)$ be the full flag in $\C_d[x]$ given by
	\[
	\mc F(\infty)=\{0\subset \C_1[x]\subset\C_2[x]\subset\dots\subset\C_d[x]\}.
	\]
	
	The subspace $X$ is a point of $\Omega_{\la}(\mc F(\infty))$ if and only if for every $i=1,\dots,N$, it
	contains a polynomial of degree $d-i-\la_{N+1-i}$. 
	
	For $z\in\C$, consider the full flag
	\[
	\mc F(z)=\{0\subset (x-z)^{d-1}\C_1[x]\subset(x-z)^{d-2}\C_2[x]\subset\dots\subset\C_d[x]\}.
	\]
	
	The subspace $X$ is a point of $\Omega_{\la}(\mc F(z))$ if and only if for every $i=1,\dots,N$, it
	contains a polynomial with a root at $z$ of order $\la_i+N-i$.
	
	A point $z\in\C$ is called a {\it base point} for a subspace $X\subset \C_d[x]$ if $g(z)=0$ for every $g\in X$.
	
	Let $\bm z=(z_1,\dots,z_n)\in {\mathring{\bP}}_n$. Let $\bLa=(\la^{(1)},\dots,\la^{(n)})$ be a sequence of partitions with at most $N$ parts. Set $|\bLa|=\sum_{s=1}^n|\la^{(s)}|$.

	Assuming $|\bLa|=N(d-N)$, denote by $\Omega_{\bLa,\bs z}$ the intersection of the Schubert cells:
	\beq\label{Omega}
	\Omega_{\bLa,\bs z}= \bigcap_{s=1}^n\Omega_{\la^{(s)}}(\mc F(z_s)).
	\eeq
	
	Note that due to our assumption, $\Omega_{\bLa,\bs z}$ is a finite subset of $\Gr(N,d)$. Moreover $\Omega_{\bLa,\bs z}$ is non-empty if and only if $\dim(V_\bLa)^{\sln_N}>0$.
	
	Define the {\it stratum} $\Omega_{\bLa}$ by the formula
	\beq\label{strata}
	\Omega_{\bLa}:=\bigcup_{\bm z\in {\mathring{\bP}}_n}\Omega_{\bLa,\bm z}\subset \Gr(N,d).
	\eeq
	
	\subsection{Self-dual spaces}\label{sec self-dual}Let $\bLa=( \la^{(1)},\dots,\la^{(n)})$ be a tuple of partitions with at most $N$ parts such that $|\bLa|=N(d-N)$ and let $\bs z=(z_1,\dots,z_n)\in \mathring{\bP}_n$.
	
	Define a sequence of polynomials $\bs T=(T_1,\dots,T_N)$ by
	\[
	T_i(x)=\prod_{s=1}^n(x-z_s)^{\la_{i}^{(s)}-\la_{i+1}^{(s)}},\quad i=1,\dots,N,
	\]
	where $\la_{N+1}^{(s)}=0$. Here and in what follows we use the convention that $x-z_s$ is considered as the constant function $1$ if $z_s=\infty$. We say that $\bs T$ is {\it associated with} $\bLa,\bm z$.
	
	Let $X\in \Omega_{\bLa,\bm z}$ and $g_1,\dots,g_i\in X$. Define the \emph{divided Wronskian $\Wr^\dag$ with respect to $\bLa,\bm z$} by
	\[
	\Wr^\dag(g_1,\dots,g_i)=\Wr(g_1,\dots,g_i)\prod_{j=1}^i T_{N+1-j}^{j-i-1},\quad i=1,\dots,N.
	\]
	Note that $\Wr^\dag(g_1,\dots,g_i)$ is a polynomial in $x$.
	
	Let $\Gamma=\{u_1,\dots,u_N\}$ be a basis of $X\in \Omega_{\bLa,\bm z}$, define a sequence of polynomials
	\beq\label{eq y}
	y_{N-i}=\Wr^\dag(u_1,\dots,u_i),\quad i=1,\dots,N-1.
	\eeq
	Denote $(y_1,\dots,y_{N-1})$ by $\bs y_\Gamma$. We say that $\bs y_\Gamma$ is {\it constructed from the basis} $\Gamma$.
	
	Given $X\in \Gr(N,d)$, define the \emph{dual space} $X^\dag$ of $X$ by
	$$ X^\dag=\{ \Wr^\dag (g_1,\dots,g_{N-1})\ | \ g_i\in X,\ i=1,\dots,N-1\}. $$
	
	Note that $X^\dag$ has no base points.	
	
	Given a space of polynomials $X$ and a rational function $g$ in $x$, denote by $g\cdot X$ the space of rational functions of the form $g\cdot f$ with $f\in X$.
	
	A space of polynomials $X$ is called a {\it pure self-dual space} if $X=X^\dag$. A space of polynomials $X$ is called \emph{self-dual} if $X = g\cdot X^\dag$ for some polynomial $g\in\C[x]$. In particular, if $X\in \Omega_{\bLa,\bm z}$ is self-dual, then $g=T_N$ (that is $X = T_N\cdot X^\dag$) and $T_{N-i}=T_i$, $i=1,\dots,N-1$, see for example Lemma 4.1 in \cite{LMV2}. 
	
	If $X$ is self-dual then $g\cdot X$ is also self-dual for any nonzero $g\in\C[x]$.
	
	The self-dual spaces are characterized by the following lemma.
	\begin{lem}\label{lem y sd}
	Let $X\in \Omega_{\bLa,\bm z}$ and let $\bm T$ be associated with $\bLa,\bm z$. Suppose $T_i=T_{N-i}$ for $i=1,\dots,N-1$. Then $X$ is self-dual if and only if there exists a basis $\Gamma=\{u_1,\dots,u_N\}$ of $X$ such that $\bm y_\Gamma=(y_1,\dots,y_{N-1})$ given by \eqref{eq y} satisfies $y_{i}=y_{N-i}$ for $i=1,\dots,N-1$.
	\end{lem}
	\begin{proof}
		The only if part has been proved in \cite{MV2}.	Since $T_i=T_{N-i}$ for $i=1,\dots,N-1$, the if part follows from Lemmas 5.2, 5.3, 5.4 of \cite{LMV2}.
	\end{proof}
		
	By Corollary 6.5 of \cite{MV2}, a pure self-dual space of polynomials $X$ possesses a nondegenerate bilinear form $\kappa$ given by
	\[
	\kappa(u_1,u_2)=\Wr^\dag(u_1,v_1,\dots,v_{N-1}),\quad \text{if }u_2=\Wr^\dag(v_1,\dots,v_{N-1}),
	\]
	where $u_i,v_j\in X$, $i=1,2$, $j=1,\dots,N-1$. This form is symmetric if $N$ is odd and skew-symmetric if $N$ is even.
	
	\subsection{Self-self-dual spaces}\label{sec ssd gr}
	From now on, we set $N=7$. We simply write $\Gr_d$ for $\Gr(7,d)$.
	
	Following Definition 4.1 of \cite{BM}, a pure self-dual space of polynomials $X\in\Omega_{\bLa,\bs z}$ of dimension 7 is called \emph{pure self-self-dual} if for a generic monic polynomial $f\in X$ there exists a three-dimensional isotropic  (with respect to the bilinear form $\kappa$) subspace $U\subset X$ such that
	$$
	f\perp U \quad  \text{and} \quad \Wr(U)=T_1^2T_2f^2.
	$$
	
	A self-dual space of polynomials $X\in\Omega_{\bLa,\bs z}$ of dimension 7 is called \emph{self-self-dual} if $X/T_7$ is pure self-self-dual.
	
	Note that if $X$ is self-self-dual then $g\cdot X$ is also self-self-dual for any nonzero polynomial $g\in\C[x]$.
	If $X\in \Omega_{\bLa,\bm z}$ is a self-self-dual space, then $T_1=T_3=T_4=T_6$ and $T_2=T_5$, see Lemma 4.8 in \cite{BM}.

	Here is another description of pure self-self-dual spaces. 
	\begin{thm}[\cite{BM}] 
	Suppose $X\in \Omega_{\bLa,\bm z}$ is a pure self-dual space and $T_1=T_3$. 
	Then $X$ is pure self-self-dual if and only if 
		\[	
		\{u^2~|~u\in X \}=\{\Wr^\dagger(u_1,u_2,u_3)~|~u_i\in X ,\, \kappa(u_i,u_j)=0\}.
		\]\qed
	\end{thm}	
	
    The self-self-dual spaces are characterized by the following lemma,  cf. Lemma \ref{lem y sd}.
		
	\begin{lem}\label{lem y ssd}
		Let $X\in \Omega_{\bLa,\bm z}\subset \Gr_d$ and let $\bm T$ be associated with $\bLa,\bm z$. Suppose $T_1=T_{3}=T_4=T_6$ and $T_2=T_5$. Then $X$ is self-self-dual if and only if there exists a basis $\Gamma=\{u_1,\dots,u_7\}$ of $X$ such that $\bm y_\Gamma=(y_1,\dots,y_{6})$  given by \eqref{eq y} satisfies $y_1^2=y_6^2=y_3=y_4$, $y_2=y_5$.
	\end{lem}
	\begin{proof}
	The only if part has been proved in \cite{BM}. 
	
	Consider the if part. Without loss of generality, we assume that $X$ has no base points. We follow the strategy in Section 5 of \cite{BM}. It follows from Lemma \ref{lem y sd} that $X$ is self-dual. By assumption, there exist polynomials $\tl y_1,\tl y_2$ such that
	\[
	\Wr(y_1,\tl y_1)=T_1y_2,\qquad \Wr(y_2,\tl y_2)=T_2y_1^3.
	\]
	For instance, since $y_1=u_1$ we can take $\tl y_1=u_2$ and $\tl y_2=\Wr^\dag(u_1,u_3)$. Set $\bm y=(y_1,y_2)$. Hence we can define descendants $(\tilde y_1,y_2)$, $(y_1,\tilde y_2)$ of $\bm y$, and the
	G$_2$-population originated at $\bm y$. Similarly, we can prove the same statement as Lemma 5.2 in \cite{BM} and the first statement of Lemma 5.3 in \cite{BM} with the modified terminology. Now the lemma can be proved in the same way as Theorem 5.4 of \cite{BM} with the following change in the first paragraph of page 108: It follows that there exist constants $c_1,c_2,c_3,c_4$ such that the C$_3$-triple $(\tl y_1,\tl y_2, \tl y_3)$
	corresponds to the flag $\tl {\mc F}$, which is related to the basis
    \[
    \{\tl u_1=u_1+c_1u_2,\,\tl u_2=u_2+c_2u_3,\,\tl u_3=u_3+c_3u_4+c_4u_5,\,\tl u_4,\,\tl u_5,\,\tl u_6,\,\tl u_7\},
    \]
	see Lemmas 6.14-6.15 in \cite{MV2}. In particular, $\Wr^\dag(\tl u_1,\tl u_2,\tl u_3)=\tl y_3^2$.
	
	This change does not alter the rest of the proof since the constant $c_4$ is not used there.
	\end{proof}

	Assume the pair $\bm y=(y_1,y_2)$ of polynomials represents a solution to the Bethe ansatz equation \eqref{eq:bae.a}, \eqref{eq:bae.b}, where all weights are dominant integral. Recall the polynomials $\mc T_1, \mc T_2$ in \eqref{eq poly T} and the differential operator $\mc D_{\bm y}$ in \eqref{eq dy}.
	
	\begin{lem}[\cite{BM}]\label{lem ker ssd}
		 The kernel of the differential operator $\mc D_{\bm y}$ is a pure self-self-dual space of polynomials.
	\end{lem}

    \begin{rem}
    	Let $X$ be the kernel of $\mc D_{\bm y}$. The corresponding polynomials $T_1,\dots, T_7$ are given by $\mc T_1=T_1=T_3=T_4=T_6$, $\mc T_2=T_2=T_5$, and $T_7=1$. Moreover, there exists a basis $\Gamma=\{u_1,\dots,u_7\}$ of $X$ such that $\bs y_{\Gamma}=(y_1,y_2,y_1^2, y_1^2,y_2,y_1)$.
    \end{rem}
	 
	For a self-self-dual space the pair $(y_1,y_2)$ in Lemma \ref{lem y ssd} can be chosen generic.
	\begin{lem}\label{lem y}
	Suppose $X\in \Omega_{\bLa,\bm z}$ is a self-self-dual space. Then for any $z\in\C$ there exists a basis $\Gamma=\{u_1,\dots,u_7\}$ of $X$ such that $\bm y_\Gamma=(y_1,\dots,y_6)$ given by \eqref{eq y} satisfies $y_1^2=y_6^2=y_3=y_4$, $y_2=y_5$, and $y_i(z)\ne 0$ for $i=1,2$. 
	\end{lem}
	\begin{proof}
	The lemma follows from the proof of Theorem 5.7 in \cite{BM}, see also Theorems 8.2 and 8.3 of \cite{MV2}.
	\end{proof}
		
		\subsection{Self-self-dual Grassmannian and the strata}\label{sec ssd structure}

	Let $\mb S\Gr_d$ denote the set of all self-self-dual spaces in $\Gr_d$. We call $\mb S\Gr_d$ the \emph{self-self-dual Grassmannian}.
	The self-self-dual Grassmannian $\mb S\Gr_d$ is an algebraic subset of $\Gr_d$. 
	
	Let  $\Omega_{\bLa,\bm z}$, $\Omega_{\bLa}$ be as in \eqref{Omega}, \eqref{strata}. Denote by $\mb S\Omega_{\bLa,\bm z}$ the set of all self-self-dual spaces in $\Omega_{\bLa,\bm z}$ and by $\mb S\Omega_{\bLa}$ the set  of all self-self-dual spaces in $\Omega_{\bLa}$: 
	\begin{align*}
	\mb S\Omega_{\bLa,\bm z}=\Omega_{\bLa,\bm z}\bigcap \mb S\Gr_d \quad \text{and} \quad \mb S\Omega_{\bLa}=\Omega_{\bLa}\bigcap \mb S\Gr_d.
	\end{align*}
	For $\bLa$ such that $|\bLa|=7(d-7)$, we call the set $\mb S\Omega_{\bLa}$ a {\it $\g$-stratum} of the self-self-dual Grassmannian.
	A stratum $\mb S\Omega_{\bLa}$ does not depend on the order of the set of partitions $\bLa$. Each $\mb S\Omega_{\bLa}$ is a constructible subset of the Grassmannian $\Gr_d$ in Zariski topology.

	Let $\mu$ be a dominant integral $\g
	$-weight and $k\in\Z_{\gge 0}$. Define a
	partition $\mu_{A,k}$ with at most $7$ parts by the rule: $(\mu_{A,k})_7=k$ and
	\[
	(\mu_{A,k})_i-(\mu_{A,k})_{i+1}=\begin{cases}
		\langle \mu,\calpha_1\rangle,\quad &\text{if }i=1,3,4,6,\\
		\langle \mu,\calpha_{2}\rangle,\quad &\text{if }i=2,5.
	\end{cases}
	\]
	We call $\mu_{A,k}$ the partition \emph{associated with weight $\mu$ and integer $k$}.
	
	Let $\bLa=(\la^{(1)},\dots,\la^{(n)})$ be a sequence of dominant integral $\g$-weights and let $\bm k=(k_1,\dots,k_n)$ be an $n$-tuple of nonnegative integers. Then denote $\bLa_{A,\bm k}=(\la_{A,k_1}^{(1)},\dots,\la_{A,k_n}^{(n)})$ the sequence of partitions associated with $\la^{(s)}$ and $k_s$, $s=1,\dots,n$. We have
	\beq\label{La size}
	|\bLa_{A,\bs k}|=7\sum_{s=1}^n (k_s+ 2\langle \la^{(s)},\calpha_1\rangle+\langle \la^{(s)},\calpha_2\rangle).
	\eeq
	
	We use the notation $\mu_{A}=\mu_{A,0}$ and $\bLa_{A}=\bLa_{A,(0,\dots,0)}$.
	
	The following lemma is clear.
	\begin{lem}\label{lem sym weight new}
		Suppose $\bm\Xi$ is a sequence of partitions with at most $N$ parts such that $|\bm{\Xi}|=7(d-7)$ and $\mb S\Omega_{\bm\Xi}$ is nonempty, then $\bm\Xi$ has the form 
		$\bm\Xi=\bLa_{A,\bm k}$ for 
		a unique sequence of dominant integral $\g$-weights $\bLa=(\la^{(1)},\dots,\la^{(n)})$ and a unique $n$-tuple $\bm k$ of nonnegative integers. \qed
	\end{lem}
	Therefore the $\g$-strata are effectively parameterized by sequences of dominant integral $\g$-weights and tuples of nonnegative integers. 
	In what follows we write $\mb S\Omega_{\bLa,\bm k}$ for $\mb S\Omega_{\bLa_{A,\bm k}}$,  $\mb S\Omega_{\bLa,\bm z}$ for  $\mb S\Omega_{\bLa_{A},\bm z}$, $\Omega_{\bLa,\bm k}$ for $\Omega_{\bLa_{A,\bm k}}$, etc. 
	
	\subsection{Self-self-dual spaces and the Gaudin model}
	Recall that the objective of the Gaudin model is to study the common eigenvectors and eigenvalues of the Bethe algebras $\mc B$ acting on $\g[t]$-modules, see Section \ref{sec:Bethe alg}. The following theorem connects self-self-dual spaces of polynomials with given ramification conditions and the Gaudin model associated to $\g$.
	
	\begin{thm}\label{bi rep ssgr}
		Let $\bLa=(\la^{(1)},\dots,\la^{(n)})$ be a sequence of dominant integral $\g$-weights, $\bm z\in \mathring{\bP}_n$, and $\bm T$ the sequence of polynomials associated with $\bLa_A,\bm z$. Then for any $\mc B$-eigenvector $v\in (V_{\bLa,\bm z})^{\g}$, we have $${\rm{Ker}}\ \big((T_1^2T_2)\cdot\mc D_v\cdot (T_1^2T_2)^{-1}\big)\in \mb S\Omega_{\bLa,\bs z}.$$
		
		Moreover, if $|\bLa_A|=7(d-7)$, then this defines a bijection between the joint eigenvalues of $\mc B$ on $(V_{\bLa,\bm z})^{\g}$ and $\mb S\Omega_{\bLa,\bs z}\subset\Gr_d$.
	\end{thm}
	\begin{proof}
		Theorem \ref{bi rep ssgr} is deduced from \cite{R} in Section \ref{sec proof G}.
	\end{proof}
	
	\begin{rem}
	It follows from Corollary 3.4 of \cite{R} that the Bethe algebra $\mc B$ has simple spectrum on $(V_{\bLa,\bm z})^\g$ for generic $\bm z$. In particular, the map $v\to {\rm{Ker}}\ \big((T_1^2T_2)\cdot\mc D_v\cdot (T_1^2T_2)^{-1}\big)$ also gives a bijection between $\mc B$-eigenvectors in $(V_{\bLa,\bm z})^\g$ and $\mb S\Omega_{\bLa,\bs z}\subset\Gr_d$ when $\bm z$ is generic. We expect this is true for all $\bm z\in\mathring{\bP}_n$.	
	\end{rem}
	
	We also have the following lemma from \cite{R}.
	
	\begin{lem}\label{lem BC completeness} Let $\bm z$  be a generic point in ${\mathring{\bP}}_n$. Then the action of the Bethe algebra $\mc B$ on $(V_{\bLa,\bm z})^{\g}$ is diagonalizable and have simple spectrum. In particular, this statement holds for any sequence $\bm z\in \R{\mathring{\bP}}_n$.\qed
	\end{lem}
	
	From Theorem \ref{bi rep ssgr}, we observe the following relation between Gaudin models associated to $\g$ and to $\gl_7$. Let $\bLa=(\la^{(1)},\dots,\la^{(n)},\la)$ be a sequence of dominant integral $\g$-weights and $\bm z=(z_1,\dots,z_n,\infty)\in\mathring{\bP}_{n+1}$. Assume that $|\bLa_A|=7(d-7)$. Set $k=d-7-2(2\la_1+\la_2)$. Consider the $\g[t]$-module $\bs V=\bigotimes_{s=1}^n V_{\la^{(s)}}(z_s)$ and the $\gl_7[t]$-module $\bs V_A=\bigotimes_{s=1}^n V_{\la_A^{(s)}}(z_s)$. Assume  $(\bs V)^\sing_{\la}\neq 0$ and that $\mc B$ acting on it is diagonalizable. Hence $\mc B$ has simple spectrum on $(\bs V)^\sing_{\la}$ by Corollary 3.4 of \cite{R}.
	
	Let $\mc B_A$ be the Bethe subalgebra of $\mc U(\gl_7[t])$, see \cite{MTV2}. Then the eigenvectors of $\mc B_A$ in $(\bs V_A)^{\sing}_{\la_{A,k}}$ (the subspace of $\gl_7$-singular vectors in $\bs V_A$ of weight $\la_{A,k}$) correspond to spaces of polynomials in $\Omega_{\bLa,\bs z}$. Each $\mathcal B$-eigenvector $v$ in $(\bs V)^\sing_\la$ corresponds to a self-self-dual space in $\mb S\Omega_{\bLa,\bs z}\subset\Omega_{\bLa,\bs z}$ and therefore to a $\mathcal B_A$-eigenvector $v_A$ in $(\bs V_A)^{\sing}_{\la_{A,k}}$. Moreover, the differential operator $T_2T_1^2\cdot\mathcal D_v \cdot T_2^{-1}T_1^{-2}$ coincides with the corresponding $\gl_7$-operator $\mc D^A_{v_A}$, see (5.1) in \cite{MTV2}. In particular, the eigenvalues of $\mc B$ on $v$ can be identified with eigenvalues of $\mc B_A$ on $v_A$. 
	
	Thus the $\mc B_A$-submodule of $(\bs V_A)^{\sing}_{\la_{A,k}}$ spanned by vectors associated to self-self-dual spaces is canonically identified with the $\mathcal B$-module $(\bs V)^\sing_\la$.
	
	This correspondence is not related to the embedding $\pi$ and we have no explanation for this phenomenon on the level of Lie algebras. This is similar to the situation in types B and C, see Theorem A.10 in \cite{LMV2}.
	
	It would be interesting to extend this correspondence to general (non-diagonalizable) case. It is also interesting to find a criterion for a $\mc B_A$-eigenvector to correspond to a self-self-dual space.
	
	\subsection{Properties of the strata}\label{sec prop strata} We describe simple properties of the strata  $\mb S\Omega_{\bLa,\bm k}$.
	
	Given $\bLa,\bs k,\bs z$, define $\tilde{\bLa},\tilde {\bm k}, \tilde {\bs z}$ by removing all zero components, 
	that is the ones with both $\la^{(s)}=0$ and $k_s=0$. 
	Then $\mb S\Omega_{\tl{\bLa},\tilde {\bs k},\tl{\bm z}}=\mb S\Omega_{\bLa,\bs k,\bm z}$ and  $\mb S\Omega_{\tl{\bLa},\tilde {\bs k}}=\mb S\Omega_{\bLa,\bs k}$. 	Also, $|\bLa_{A,\bm k}|$ is divisible by $7$.

	We say that $(\bLa,\bm k)$ is $d$-\emph{nontrivial} if and only if $(V_\bLa)^{\g} \ne 0$, $|\la_{A,k_s}^{(s)}|>0$, $s=1,\dots,n$, and $|\bLa_{A,\bm k}|=7(d-7)$. 
	
	Theorem \ref{bi rep ssgr} implies that if $(\bLa,\bm k)$ is $d$-nontrivial then the corresponding stratum $\mb S\Omega_{\bs \La,\bm k}\subset \mb S\Gr_d$ is nonempty and, conversely, that if $(\bLa,\bm k)$ is a nonempty stratum of $\mb S\Gr_d$ then the pair $(\tilde{\bLa},\tilde {\bm k})$  obtained by removing all zero components, is $d$-nontrivial. 

	\medskip  
	
	Note that $|\bLa_{A,\bm k}|=|\bLa_{A}|+7|\bs k|$, where $|\bs k|=k_1+\dots +k_n$. In particular, if $(\bLa,\bs 0)$ is $d$-nontrivial then $(\bLa,\bs k)$ is $(d+|\bs k|)$-nontrivial.
	Further, there exists a bijection between $\Omega_{\bLa,\bm z}$ in $\Gr_d$ and $\Omega_{\bLa,\bm k,\bm z}$ in $\Gr_{d+|\bs k|}$ given by
	\beq\label{mult}
	\Omega_{\bLa,\bm z}\rightarrow \Omega_{\bLa,\bm k,\bm z},\quad X \mapsto \prod_{s=1}^n(x-z_s)^{k_s}\cdot X.
	\eeq
	Moreover, \eqref{mult} restricts to a bijection between $\mb S\Omega_{\bLa,\bm z}\subset\mb S\Gr_d$ and $\mb S\Omega_{\bLa,\bm k,\bm z}\subset \mb S\Gr_{d+|\bs k|}$. 
	
	\medskip
	
	Let $\mu^{(1)},\dots,\mu^{(a)}$ be all distinct partitions in $\bLa_{A,\bm k}$. Let $n_i$ be the number of occurrences of ${\mu}^{(i)}$ in $\bLa_{A,\bm k}$, then $\sum_{i=1}^an_i=n$. Denote $\bm n=(n_1,\dots,n_a)$. We write $\bLa_{A,\bm k}$ in the following order: $\la_{A,k_i}^{(i)}=\mu^{(j)}$ for $\sum_{s=1}^{j-1}n_s+1\lle i\lle \sum_{s=1}^{j}n_s$, $j=1,\dots,a$.
	
	 Let $\mathfrak S_{\bm n;n_i}$ be the subgroup of the symmetric group $\mathfrak S_n$ permuting $\{n_1+\cdots+n_{i-1}+1,\dots,n_1+\cdots+n_{i}\}$, $i=1,\dots,a$. Then the group $\mathfrak S_{\bm n}=\mathfrak S_{\bm n;n_1}\times \mathfrak S_{\bm n;n_2}\times \dots\times \mathfrak S_{\bm n;n_a}$ acts freely on ${\mathring{\bP}}_n$ and on $\R{\mathring{\bP}}_n$. Denote by ${\mathring{\bP}}_n/\mathfrak S_{\bm n}$ and $\R{\mathring{\bP}}_n/\mathfrak S_{\bm n}$ the sets of orbits. 
	\begin{prop}\label{prop bij BC}
		Suppose $(\bLa,\bm k)$ is $d$-nontrivial. The set $\mb S\Omega_{\bLa,\bm k}$ is a ramified covering of ${\mathring{\bP}}_n/\mathfrak S_{\bm n}$.   Moreover, the degree of the covering is equal to $\dim(V_\bLa)^{\g}$.  In particular, $\dim \mb S\Omega_{\bLa,\bm k}=n$. Over $\R{\mathring{\bP}}_n/\mathfrak S_{\bm n}$, this covering is unramified of the same degree, moreover all points in fibers are real. 
	\end{prop}
	\begin{proof}
		The proposition follows from the combination of Theorem \ref{bi rep ssgr}, Lemma \ref{lem BC completeness}, and Theorem 1.1 of \cite{MTV1}.
	\end{proof}

	We find the strata $\mb S\Omega_{\bLa,\bm k}\subset\mb S\Gr_d$ of the largest dimension. 
	\begin{lem}\label{thm dim sgr}
		The $d$-nontrivial strata $\mb S\Omega_{\bs \La,\bm k}\subset\mb S\Gr_d$ with the largest dimension have $(\la^{(s)},k_s)$ equal to either $(\omega_2,0)$ or $(0,1)$, $s=1,\dots,d-7$. Each such stratum is nonempty if the number of $s\in\{1,\dots,d-7\}$ with $(\la^{(s)},k_s)=(\omega_2,0)$ is not $1$, and then it has dimension $d-7$. In particular, there is at least one nonempty stratum of this dimension, and if $d>8$ then more than one.
	\end{lem}
	\begin{proof}
		The lemma follows from formula \eqref{La size} taking into account the fact that $V_0$ and $V_{\omega_2}$ are both submodules of  $V_{\omega_2}\otimes V_{\omega_2}$, see \eqref{eq:dec of B}. 
	\end{proof}
	
	\subsection{The $\g$-stratification of the self-self-dual Grassmannian}
	The following theorem follows directly from Theorem 3.5 of \cite{LMV2} and Theorem \ref{bi rep ssgr}. 
	\begin{thm}\label{thm sgr dec}
		We have
		$$
		\mb S\Gr_d=\bigsqcup_{d\text{-nontrivial }(\bLa,\bm k)}\mb S\Omega_{\bLa,\bm k}.
		$$
		\qed
	\end{thm}
	
	For a $d$-nontrivial $(\bLa,\bm k)$, we call the closure of $\mb S\Omega_{\bLa,\bm k}$ inside $\mb S\Gr_d$, a {\it $\g$-cycle}.
	The $\g$-cycles $\overline{\mb S\Omega}_{\bLa,\bm k}$ are algebraic sets in $\mb S\Gr_d$ and therefore in $\Gr_d$. We 
	describe  $\g$-cycles as unions of $\g$-strata.
	
	Define a partial order $\gge$ on the set of pairs $\{(\bLa,\bm k)\}$ as follows.
	Let $\bLa=(\la^{(1)},\dots,\la^{(n)})$, $\bs\Xi=(\xi^{(1)},\dots,\xi^{(m)})$ be two sequences of dominant integral $\g$-weights. Let $\bm k=(k_1,\dots,k_n)$, $\bm r=(r_1,\dots,r_m)$ be two tuples of nonnegative integers. We say that $(\bLa,\bm k)\gge (\bm\Xi,\bm r)$ if and only if there exists a partition $\{I_1,\dots,I_m\}$ of $\{1,2,\dots,n\}$ such that
	\[
	{ \mathrm{Hom}}_{\g}(V_{\xi^{(i)}},\bigotimes_{j\in I_i} V_{\la^{(j)}})\ne 0,\qquad |\xi^{(i)}_{A,r_i}|=\sum_{j\in I_i}|\la_{A,k_j}^{(j)}|,
	\]for $i=1,\dots,m$.
	
	If $(\bLa,\bm k)\gge (\bm\Xi,\bm r)$ are $d$-nontrivial, we call $\mb S\Omega_{\bm\Xi,\bm r}$ a \emph{degeneration} of $\mb S\Omega_{\bLa,\bm k}$. In the case of $m=n-1$, we call $\mb S\Omega_{\bm\Xi,\bm r}$ a \emph{simple degeneration} of $\mb S\Omega_{\bLa,\bm k}$. 
	
	\begin{thm}\label{thm simple deg G} 
		If $\mb S\Omega_{\bm\Xi,\bm r}$ is a degeneration of $\mb S\Omega_{\bLa,\bm k}$, then $\mb S\Omega_{\bm\Xi,\bm r}$ is contained in the $\g$-cycle $\overline{\mb S\Omega}_{\bLa,\bm k}$.
	\end{thm}
	\begin{proof}
	Assuming Theorem \ref{bi rep ssgr}, Theorem \ref{thm simple deg G} is proved in a similar way as Theorem 3.6 of \cite{LMV2}.	\end{proof}
	
	\begin{thm}\label{thm G strata}
		For $d$-nontrivial $(\bLa,\bm k)$, we have
		$$
		\overline{\mb S\Omega}_{\bLa,\bm k}=\bigsqcup_{\substack{(\bm\Xi,\bm r)\lle (\bLa,\bm k),\\ d\text{-nontrivial }(\bm\Xi,\bm r)}}\mb S\Omega_{\bm\Xi,\bm r}.
		$$
	\end{thm}
	\begin{proof}
		Theorem \ref{thm G strata} follows from Theorem \ref{thm simple deg G}.
	\end{proof}
	\medskip
	
	Theorems \ref{thm sgr dec} and \ref{thm G strata} imply that the subsets $\mb S\Omega_{\bLa,\bm k}$ with $d$-nontrivial $(\bLa,\bm k)$ give a stratification of $\mb S\Gr_d$, similar to the $\gl_{N}$-stratification of $\Gr(N,d)$ and $\g_N$-stratification of $\sGr(N,d)$ described in \cite{LMV2}. We call it the {\it $\g$-stratification of $\mb S\Gr_d$}. 
	
	An example of $\g$-stratification of $\mb S\Gr_{11}$ is given in Appendix \ref{app ex}.
	
	\begin{rem}\label{translations}
		The group of affine translations acts on $\C_d[x]$ by changes of variable. Namely, for  $a\in\C^\times,b\in\C$, we have a map sending
		$f(x)\mapsto f(ax+b)$ for all $f(x)\in\C_d[x]$. This group action preserves the self-self-dual Grassmannian $\mb S\Gr_d$ and the $\g$-strata $\mb S\Omega_{\bLa}$.
	\end{rem}
	
	\subsection{The $\g$-stratification of $\mb S\Gr_d$ and the Wronski map}
	Let $\bLa=(\la^{(1)},\dots,\la^{(n)})$ be a sequence of dominant integral $\g$-weights and let $\bm k=(k_1,\dots,k_n)$ be an $n$-tuple of nonnegative integers. Let $\bs z=(z_1,\dots,z_n)\in \mathring{\bP}_n$.
	
	Recall that $\la^{(s)}_i=\langle \la^{(s)},\calpha_i\rangle$, $i=1,2$. If $X\in\mb S\Omega_{\bLa,\bm k,\bm z}$, one has
	\[
	\Wr(X)=\prod_{s=1}^n(x-z_s)^{7(2\la^{(s)}_1+\la^{(s)}_{2}+k_s)}.
	\]
	
	Define the \emph{reduced Wronski map} $\overline{\Wr}$ as follows. The map
	\[
	\overline{\Wr}:\mb S\Gr_d\to \Gr(1,d-6)
	\]is sending $X\in \mb S\Gr_d$ to $\C(\Wr(X))^{1/7}$. The reduced Wronski map is a finite map. 
	
	Let $\bs m=(m^{(1)},\dots,m^{(n)})$ be an unordered sequence of positive integers such that $|\bs m|=d-7$. Then we have the following proposition.
	
	\begin{prop}\label{B preimage}
		The preimage of $\Omega_{\bs m}$ in $\Gr(1,d-6)$ under the reduced Wronski map is a union of all strata $\mb S\Omega_{\bLa,\bm k}$ in $\mb S\Gr_d$ such that $|\la^{(s)}_{A,k_s}|=7m^{(s)}$, $s=1,\dots,n$. \qed
	\end{prop}
	
	Let  $\bLa=(\la^{(1)},\dots,\la^{(n)})$ be a sequence of dominant integral $\g$-weights and $\bm k=(k_1,\dots,k_n)$ a sequence of nonnegative integers.
	Let $a$ be the number of distinct pairs in the set $\{(\la^{(s)},k_s),\ s=1,\dots,n\}$. We can assume that $(\la^{(1)},k_1),\dots,(\la^{(a)},k_a)$ are all distinct, and let $n_1,\dots,n_a$ be their multiplicities, $n_1+\dots+n_a=n$. Define the {\it symmetry coefficient} of $(\bLa,\bm k)$ as the product of multinomial coefficients:
	\[
	b(\bLa,\bm k) = \prod_m \frac{\left(\sum_{s=1,\dots,a, \ |\la^{(s)}_{A,k_s}|=m} n_s\right)! }  {\prod_{s=1,\dots,a,\ |\la^{(s)}_{A,k_s}|=m} (n_s)!}.
	\]
	
	Consider a sequence of integers $\bs m=(m^{(1)},\dots,m^{(n)})$,
	where $7m^{(s)}=|\la^{(s)}_{A,k_s}|$. Consider the stratum $\Omega_{\bs m}$ in $\Gr(1,d-6)$, corresponding to polynomials with $n$ distinct roots of multiplicities $m^{(1)},\dots,m^{(n)}$. 
	
	\begin{prop}
		\label{prop cov}
		Let $(\bLa,\bs k)$ be $d$-nontrivial. Then the reduced Wronski map 
		$\overline{\Wr}|_{\mb S\Omega_{\bLa,\bs k}} : \mb S\Omega_{\bLa,\bs k}\to \Omega_{\bs m}$ is a ramified covering
		of degree $b(\bLa,\bm k) \dim (V_{\bLa})^{\g}$.  
	\end{prop}
	
	\begin{proof}
		The statement follows from Theorem \ref{bi rep ssgr}, Lemma \ref{lem BC completeness}, and Proposition \ref{B preimage}.
	\end{proof}
	
	In other words, the $\g$-stratification of $\mb S\Gr_d$ given by Theorems \ref{thm sgr dec} and \ref{thm G strata}, is adjacent to the swallowtail $\gl_1$-stratification of $\Gr(1,d-6)$ and the reduced Wronski map. 
	
	\section{Proof of Theorem \ref{bi rep ssgr}}\label{sec oper}
	In this section, we recall notation for $^t\g$-opers, and then discuss the relation between $^t\g$-opers with prescribed regular singularities and self-self-dual spaces. Then we prove Theorem \ref{bi rep ssgr} by using the bijection between joint eigenvalues of the Bethe algebra $\mc B$ and $^t\g$-opers described in \cite{R}. 
	\subsection{Lie algebra $^t\g$}
	Let $\g$ and $\h$ be as in Section \ref{sec sla}. Let $^t\g=\g({}^tA)$ be the Langlands dual Lie algebra of $\g$, then in our case $^t\g\cong\g$. 
A system of simple roots of $^t\g$ is $\calpha_1,\calpha_2$ with the corresponding coroots $\alpha_1,\alpha_2$. A weight $\la\in\h^*$ of $\g$ can be identified with a coweight of $^t\g$. Note that, under the identification of $^t\g\cong\g$, $\calpha_1$ becomes the short root of $^t\g$ while $\calpha_2$ the long one.
	
	Let $E_1,E_2$, $\alpha_1,\alpha_2$, $F_1,F_2$ be the Chevalley generators of $^t\g$. We fix the isomorphism $\iota:{}^t\g\to \g$ sending $E_i\mapsto e_{3-i}$, $F_i\mapsto f_{3-i}$ and $\alpha_i\mapsto\calpha_{3-i}$, $i=1,2$.
	
	One has the Cartan decomposition $^t\g={}^t\n_-\oplus{}^t\h\oplus{}^t\n_+$. Introduce also the positive and negative Borel subalgebras $^t\mathfrak b={}^t\h\oplus{}^t\n_+$ and $^t\mathfrak b_-= {}^t\h\oplus {}^t\n_-$.
	
	Let $\mathscr G$ be a simple Lie group of type G$_2$, $\mathscr B$ a Borel subgroup, and $\mathscr N=[\mathscr B,\mathscr B]$ its unipotent radical, with the corresponding Lie algebras ${}^t\g\supset{}^t\mathfrak b\supset{}^t\n_+$. The group $\mathscr G$ acts on $^t\g$ by adjoint action.
	
	 Let $p_{-1}=F_1+F_2$ be the regular nilpotent element. The set $p_{-1}+{}^t\mathfrak b=\{p_{-1}+b~|~b\in{}^t\mathfrak b\}$ is invariant under the adjoint action of $\mathscr N$. Consider the quotient space $(p_{-1}+{}^t\mathfrak b)/ \mathscr N$ and denote the $\mathscr N$-orbit of $g\in p_{-1}+{}^t\mathfrak b$ by $[g]_{^t\g}$.
	
	The action of $W$ on $\h=(^t\h)^*$ is given by $s_i(\cmu)=\cmu-\langle \alpha_i,\cmu\rangle \calpha_i$ for $\cmu\in\h$. We use the notation
	\[ \sigma\cdot\cla=\sigma(\cla+\crho)-\crho,\quad \sigma\in W,\quad \cla\in\h,\]
	for the shifted action of the Weyl group on $\h$. Here $\crho=5\calpha_1+3\calpha_2$.
	
	For a vector space $X$ we denote by $\cM(X)$ the space of $X$-valued meromorphic functions on $\mathbb P^1$. For a group $R$ we denote by $R(\cM)$ the group of $R$-valued meromorphic functions on $\bP^1$.
	
	\medskip

	Consider a $\g$-weight $\la$ as a $^t\g$-coweight, then $\la\in {}^t\h$. Moreover, note that
	\[\la=(2\la_1+\la_2)\alpha_1+(3\la_1+2\la_2)\alpha_2,\]
	
	We use the $7$-dimensional irreducible representation $\pi\circ \iota:  {}^t\g\to \End(\C^7)$
	to write elements of $^t\g$ as matrices. In particular, $\la$ is given by
	\[
	\mathrm{diag}(2\la_1+\la_2,\ \la_1+\la_2,\ \la_1,\ 0,\ -\la_1,\ -\la_1-\la_2,\ -2\la_1-\la_2).
	\]
	
	\subsection{Miura $^t\g$-opers}\label{sec Miura}
	We recall the basic facts about $^t\g$-opers, see \cite{F2} for more detail.
	
	Fix a global coordinate $x$ on $\C\subset\bP^1$. Consider the following subset of differential operators
	\[
	\text{op}_{^t\g}(\bP^1)=\{\pa_x+p_{-1}+\bs v~|~ \bs v\in\cM(^t\mathfrak b)\}.
	\]This set is stable under the gauge action of the unipotent subgroup $\mathscr N(\cM)\subset \mathscr G(\cM)$. The space of ${^t\g}$-{\it opers} is defined as the quotient space $\Opg(\bP^1):=\opg(\bP^1)/ \mathscr N(\cM)$. We denote by $[\nabla]$ the class of $\nabla\in \opg(\bP^1)$ in $\Opg(\bP^1)$.
	
	We say that $\nabla=\pa_x+p_{-1}+\bs v\in\opg(\bP^1)$ is {\it regular} at $z\in \C$ if $\bs v$ has no pole at $z$. A $^t\g$-oper $[\nabla]$ is said to be {\it regular} at $z\in \C$ if there exists $f\in\mathscr N(\cM)$ such that $f^{-1}\cdot\nabla\cdot f$ is regular at $z$.
	
	Let $\nabla=\pa_x+p_{-1}+\bs v$ be a representative of a $^t\g$-oper $[\nabla]$. Consider $\nabla$ as a $\mathscr G$-connection on the trivial principal bundle $p:\mathscr G\times \bP^1\to \bP^1$. The connection has singularities at the set $\mathrm{Sing}\subset \C$ where the function $\bs v$ has poles (and maybe at infinity). Parallel translations with respect to the connection define the monodromy representation $\pi_1(\C\setminus\mathrm{Sing})\to \mathscr G$. Its image is called the {\it monodromy group} of $\nabla$. If the monodromy group of one of the representatives of $[\nabla]$ is $\{I\}\subset\mathscr G$, we say that $[\nabla]$ is a {\it monodromy-free} $^t\g$-oper.
	
	A {\it Miura $^t\g$-oper} is a differential operator of the form $\nabla=\pa_x+p_{-1}+\bs v$, where $\bs v\in\cM(^t\h)$. 
	
	A $^t\g$-oper $[\nabla]$ has {\it regular singularity} at $z\in\C$, if there exists a representative $\nabla$ of $[\nabla]$ such that
	\[
	(x-z)^{\rho}\cdot\nabla\cdot (x-z)^{-\rho}=\pa_x+\frac{p_{-1}+\bs w}{x-z},
	\]where $\bs w\in\cM(^t\mathfrak b)$ is regular at $z$. The {\it residue of $[\nabla]$ at $z$} is $[p_{-1}+\bs w(z)]_{^t\g}$. We denote the residue of $[\nabla]$ at $z$ by $\mathrm{res}_{z}[\nabla]$.
	
	Similarly, a $^t\g$-oper $[\nabla]$ has {\it regular singularity} at $\infty\in\bP^1$, if there exists a representative $\nabla$ of $[\nabla]$ such that
	\[
	x^{\rho}\cdot\nabla\cdot x^{-\rho}=\pa_x+\frac{p_{-1}+\bs {\tilde w} }{x},
	\]where $\bs{\tilde w}\in\cM(^t\mathfrak b)$ is regular at $\infty$. The {\it residue of $[\nabla]$ at $\infty$} is $-[p_{-1}+\bs {\tilde w}(\infty)]_{^t\g}$. We denote the residue of $[\nabla]$ at $\infty$ by $\mathrm{res}_{\infty}[\nabla]$.
	
	\begin{lem}\label{lem res id}
    Any orbit $[p_{-1}+b]_{^t\g}$, $b\in {}^t\mathfrak{b}$, contains a representative of the form $p_{-1}-\rho-\la$ for $\la\in\h^*$. For any $\la,\mu\in\h^*$, we have $[p_{-1}-\rho-\la]_{^t\g}=[p_{-1}-\rho-\mu]_{^t\g}$ if and only if there exists $\sigma\in W$ such that $\la=\sigma\cdot \mu$. \qed
	\end{lem}
	
 By an abuse of notation we will write $[\la]_{ W}$ for $[p_{-1}-\rho-\la]_{^t\g}$. In particular, if $[\nabla]$ is regular at $z$, then $\mathrm{res}_z[\nabla]=[0]_{ W}$.
	
	Let $\bLa=(\la^{(1)},\dots,\la^{(n)},\la)$ be a sequence of $n+1$ dominant integral $\g$-weights and let $\bs z=(z_1,\dots,z_n,\infty)\in {\mathring{\bP}_{n+1}}$. Let $\text{Op}_{^t\g}(\mathbb{P}^1)_{\bLa,\bm{z}}^{\text{RS}}$ denote the set of all ${^t\g}$-opers with at most regular singularities at points $z_s$ and $\infty$ whose residues are given by
	\[
	\text{res}_{z_s}[\nabla]=[\la^{(s)}]_{ W},\quad \text{res}_{\infty}[\nabla]=-[\la]_{ W},\quad s=1,\dots,n,
	\]and which are regular elsewhere. Let $\text{Op}_{^t\g}(\mathbb{P}^1)_{\bLa,\bm{z}}\subset \text{Op}_{^t\g}(\mathbb{P}^1)_{\bLa,\bm{z}}^{\text{RS}}$ denote the subset consisting of those ${^t\g}$-opers which are also monodromy-free.
	
	\begin{lem}[\cite{F2}]\label{lem sur}
		For every $\g$-oper $[\nabla]\in\Opg(\bP^1)_{\bLa,\bm{z}}$, there exists a Miura ${^t\g}$-oper as one of its representatives.\qed
	\end{lem}
	
	\begin{lem}[\cite{F2}]\label{lem m-form}
		Let $\nabla$ be a Miura ${^t\g}$-oper, then $[\nabla]\in\Opg(\bP^1)_{\bLa,\bm{z}}^{\mathrm{RS}}$ if and only if the following conditions hold:
		\begin{enumerate}
			\item $\nabla$ is of the form
			\[
			\nabla=\pa_x+p_{-1}-\sum_{s=1}^n\frac{\sigma_s\cdot\la^{(s)}}{x-z_s}-\sum_{j=1}^m\frac{\tl \sigma_j\cdot 0}{x-t_j}
			\]
			for some $m\in\Z_{\gge 0}$, $\sigma_s\in W$ for $s=1,\dots,n$ and $ \tl \sigma_j\in  W$, $t_j\in\bP^1\setminus  \bm z$ for $j=1,\dots,m$,
			\item there exists $\sigma_\infty\in W$ such that
			\[
			\sum_{s=1}^n \sigma_s\cdot\la^{(s)}+\sum_{j=1}^m \tl \sigma_j\cdot 0=\sigma_\infty\cdot \la,
			\]
			\item $[\nabla]$ is regular at $t_j$ for $j=1,\dots,m$.
		\end{enumerate}\qed
	\end{lem}
	
	\subsection{Miura transformation}
	Using the representation $\pi\circ \iota$ of $^t\g$, one can consider a Miura $^t\g$-oper $\nabla=\pa_x+p_{-1}+\bm v(x)$  as a differential operator with coefficients in $7\times 7$-matrices. Note that $\bm v(x)$ has the following matrix form
	\[
	\bm v(x)={\rm diag}(2v_1(x)+v_2(x),v_1(x)+v_2(x),v_1(x),0,-v_1(x),-v_1(x)-v_2(x),-2v_1(x)-v_2(x)),
	\]
	where $v_i(x)=\langle \bs v(x),\calpha_i\rangle$.
	
	In a standard way, we convert the $7\times 7$-matrix  differential operator $\nabla=\pa_x+p_{-1}+\bm v(x)$ of order $1$ to a scalar linear differential operator of order $7$:
	\beq\label{eq:miura}
	L_\nabla=(\pa_x-2v_1-v_2)(\pa_x-v_1-v_2)(\pa_x-v_1)\pa_x(\pa_x+v_1)(\pa_x+v_1+v_2)(\pa_x+2v_1+v_2).
	\eeq
	Let $\bm u(x)=(u_1(x),u_2(x),\dots,u_7(x))$ be a solution of $\nabla$, then $L_\nabla u_7=0$.
	Clearly, different Miura representatives of $[\nabla]$ give the same differential operator, we have a well-defined scalar differential operator $L_{[\nabla]}$.

	The following lemma follows from the definitions.
	\begin{lem}\label{lem monodromy}
		Suppose $\nabla$ is a Miura $^t\g$-oper with $[\nabla]\in\Opg(\bP^1)_{\bLa,\bm{z}}$. Then $L_{[\nabla]}$ is a monodromy-free scalar differential operator.\qed
	\end{lem}
	
	The following lemma describes the singularities of the scalar differential operator $L_{[\nabla]}$.
	\begin{lem}\label{lem exps}
		Suppose $\nabla$ is a Miura $^t\g$-oper with $[\nabla]\in\Opg(\bP^1)_{\bLa,\bm{z}}$, then $L_{[\nabla]}$ is a monic Fuchsian differential operator with singularities at $z_s$, $s=1,\dots,n$, and $\infty$ only. The exponents of $L_{[\nabla]}$ at $z_s$ are given by \eqref{exp at z} and at $\infty$ by \eqref{exp at inf}.
	\end{lem}
	\begin{proof}
		Note that $\nabla$ satisfies the conditions (i)-(iii) in Lemma \ref{lem m-form}. By Lemma \ref{lem res id} and Theorem 5.11 in \cite{F2}, we can assume $\sigma_s=1$ for given $s$. The lemma follows directly.
	\end{proof}
	
	\subsection{Relations with pure self-self-dual spaces}
	Let $\bLa=(\la^{(1)},\dots,\la^{(n)},\la)$ be a sequence of $n+1$ dominant integral $\g$-weights and let $\bs z=(z_1,\dots,z_n,\infty)\in {\mathring{\bP}_{n+1}}$.
	
	Choose $d$ large enough so that 
	\beq\label{k}
	k:=d-7-\sum_{s=1}^n(2\la_1^{(s)}+\la_1^{(s)})-(2\la_1+\la_2)\gge 0.
	\eeq
	Let $\bm k=(0,\dots,0,k)$. Note that we always have $|\bLa_{A,\bm k}|=7(d-7)$ and spaces of polynomials in $\mb S\Omega_{\bLa,\bm k,\bm z}$ ($=\mb S\Omega_{\bLa_{A,\bm k},\bm z}$) are pure self-self-dual.
	
	Let $\bs T=(T_1,\dots,T_{7})$ be associated with $\bLa_{A,\bm k},\bs z$, then $T_7=1$, $T_1=T_3=T_4=T_6$, and $T_2=T_5$.
	\begin{thm}\label{thm bij oper sd}
		There exists a bijection between $\Opg(\mathbb{P}^1)_{\bLa,\bm{z}}$ and $\mb S\Omega_{\bLa,\bm k,\bm z}$ given by the map $[\nabla]\mapsto \Ker (T_1^2T_2\cdot L_{[\nabla]}\cdot T_1^{-2}T_2^{-1})$.
	\end{thm}
	\begin{proof}
		Given $[\nabla]\in \Opg(\mathbb{P}^1)_{\bLa,\bm{z}}$,  using an argument similar to the proof of Theorem 6.7 in \cite{LMV2}, one shows that $\Ker (T_1^2T_2\cdot L_{[\nabla]}\cdot T_1^{-2}T_2^{-1})$ is a pure self-dual space in $\Omega_{\bLa,\bm k,\bs z}$. Note the special form of $L_{[\nabla]}$ in \eqref{eq:miura}. It follows from Lemma \ref{lem y ssd} that $\Ker (T_1^2T_2\cdot L_{[\nabla]}\cdot T_1^{-2}T_2^{-1})$ is a pure self-self-dual space.
		
		Conversely, let $X\in{\mb S}\Omega_{\bLa,\bm k,\bs z}$. With the help of Lemmas \ref{lem y}, \ref{lem monodromy}, \ref{lem exps} and Theorem 4.1 of \cite{MV5}, one can construct a $^t\g$-oper $[\nabla_X]\in \Opg(\mathbb{P}^1)_{\bLa,\bm{z}}$ in a similar way as in the proof of Theorem 6.7 of \cite{LMV2}. 
	\end{proof}
	
	\subsection{Proof of Theorem \ref{bi rep ssgr}}\label{sec proof G}
	Let $\bLa=(\la^{(1)},\dots,\la^{(n)},\la)$ be a sequence of dominant integral $\g$-weights. Consider the $\g$-module
	$V_{\bLa}=V_{\la^{(1)}}\otimes\dots\otimes V_{\la^{(n)}}\otimes V_{\la}$. It follows from Theorem 3.2 and Corollary 3.3 of \cite{R} that there exists a bijection between joint eigenvalues of the Bethe algebra $\mc B$ on $(V_{\la^{(1)}}(z_1)\otimes\dots\otimes V_{\la^{(n)}}(z_n))_\la^\sing$ and the $^t\g$-opers in $\Opg(\mathbb{P}^1)_{\bLa,\bm{z}}$.

	\begin{proof} [Proof of Theorem \ref{bi rep ssgr}.]
    In the setup of Theorem \ref{bi rep ssgr}, $k$ given by \eqref{k} is zero.
	
	If $v$ is a Bethe vector in $(V_{\bLa,\bs z})^\g$, then $\Ker (T_1^2T_2\cdot \mc D_v\cdot T_1^{-2}T_2^{-1})\in \mb S\Omega_{\bLa,\bs z}$ by Lemma \ref{lem ker ssd} and Proposition \ref{prop diff oper}. The Bethe ansatz for $(\bigotimes _{s=1}^m V_{\omega_2}(\tl z_s))^\sing$ is complete for generic $(\tl z_1,\dots,\tl z_m)$ by Theorem \ref{thm G completeness}. Any finite-dimensional irreducible $\g$-module is a submodule in a tensor power of 
	$V_{\omega_2}$. The self-self-dual Grassmannian $\mb S \Gr_d$ is a closed set. Therefore, using Lemma \ref{exponents}, we deduce that $\Ker (T_1^2T_2\cdot \mc D_v\cdot T_1^{-2}T_2^{-1})\in \mb S\Omega_{\bLa,\bs z}$ holds in general.
	 
	The second statement follows from Lemma \ref{B2 and B6}, Theorem \ref{thm bij oper sd}, and Corollary 3.3 of \cite{R}.
	\end{proof}
	
    \begin{appendix}
	\section{The formula for solutions}
	\label{sec:solutions}
	In this section, we give explicit formulas for the solutions of the Bethe ansatz equation associated to data $\bs \La=(\la,\omega_2)$, $\bs z=(0,1)$ and admissible $\bm l_i$, $i=0,1,\dots,6$, see \eqref{adm l}.
	
	We use $\bm {y^{l_i}}=(y_1^{(i)},y_2^{(i)})$ to represent the solution of the Bethe ansatz equation associated to $\bLa,\bs z, \bm l_i$.
	
	{\bf The case of $i=0$}.
	This case is trivial. We have $\bm {y^{l_0}}=(1,1)$.

	\medskip
	
	{\bf The case of $i=1$}.
	We have $$ y_1^{(1)}=1,\qquad y_2^{(1)}=x-\dfrac{\la_2}{\la_2+1}.$$

	{\bf The case of $i=2$}.
	We have
	\begin{align*} 
	& y_1^{(2)}=x-\dfrac{\la_1(3\la_1+\la_2+3)}{(\la_1+1)(3\la_1+\la_2+4)},\\ 
	 & y_2^{(2)}=x-\dfrac{3\la_1+\la_2+3}{3\la_1+\la_2+4}.
	\end{align*}
		
	{\bf The case of $i=3$}.
	We have  
	\[
	y_1^{(3)}=x-\dfrac{(3\la_1+\la_2+3)(3\la_1+2\la_2+4)}{(3\la_1+\la_2+5)(3\la_1+2\la_2+6)},\hspace{175pt}
	\] 

	\begin{align*}
	y_2^{(3)}=x^2-&\dfrac{2(3\la_1+2\la_2+4)(3\la_1+5\la_2+3\la_1\la_2+\la_2^2+6)}{(\la_2+2)(3\la_1+\la_2+5)(3\la_1+2\la_2+6)}x\\&\qquad\qquad\qquad\qquad\qquad+\frac{\la_2(3\la_1+\la_2+3)(3\la_1+2\la_2+4)^2}{(\la_2+2)(3\la_1+\la_2+5)(3\la_1+2\la_2+6)^2}.
	\end{align*}
			
	{\bf The case of $i=4$}.
	We have  
	\[
	y_1^{(4)}=x-\dfrac{(\la_1+\la_2+1)(3\la_1+2\la_2+4)}{(\la_1+\la_2+2)(3\la_1+2\la_2+5)},\hspace{200pt}
	\]

	\begin{align*}
	y_2^{(4)}=\,&x^3-\dfrac{3(\la_1+\la_2+1)(3\la_1\la_2+2\la_2^2-2\la_1+3\la_2-2)}{\la_2(\la_1+\la_2+2)(3\la_1+2\la_2+5)}x^2\\&+\frac{3(\la_2-1)(\la_1+\la_2+1)^2(3\la_1+2\la_2+4)(3\la_1\la_2+2\la_2^2+2\la_1+6\la_2+4)}{\la_2(\la_2+1)(\la_1+\la_2+2)^2(3\la_1+2\la_2+5)^2}x\\&-\frac{(\la_2-1)(\la_1+\la_2+1)^3(3\la_1+2\la_2+4)^2}{(\la_2+1)(\la_1+\la_2+2)^3(3\la_1+2\la_2+5)^2}.
	\end{align*}	
		
	{\bf The case of $i=5$}.
	We have  
	\begin{align*}
	y_1^{(5)}=&\,x^2-\dfrac{2(2\la_1+\la_2+2)(3\la_1+\la_2+2)(3\la_1^2+2\la_1\la_2+6\la_1+\la_2+3)}{(\la_1+1)(2\la_1+\la_2+3)(3\la_1+\la_2+3)(3\la_1+2\la_2+5)}x\hspace{100pt}\\&\hspace{120pt}+\frac{\la_1(2\la_1+\la_2+2)^2(3\la_1+\la_2+2)(3\la_1+2\la_2+4)}{(\la_1+1)(2\la_1+\la_2+3)^2(3\la_1+\la_2+4)(3\la_1+2\la_2+5)},
	\end{align*}
	
	\begin{align*}
	y_2^{(5)}&=x^3-\dfrac{3(2\la_1+\la_2+2)(9\la_1^2+9\la_1\la_2+2\la_2^2+20\la_1+9\la_2+11)}{(2\la_1+\la_2+3)(3\la_1+\la_2+3)(3\la_1+2\la_2+5)}x^2\\+&\frac{3(3\la_1+\la_2+2)(2\la_1+\la_2+2)^2(3\la_1+2\la_2+4)(9\la_1^2+9\la_1\la_2+2\la_2^2+25\la_1+12\la_2+18)}{(2\la_1+\la_2+3)^2(3\la_1+\la_2+3)(3\la_1+\la_2+4)(3\la_1+2\la_2+5)^2}x\\&\qquad\qquad\qquad\qquad\qquad\qquad\qquad\qquad-\frac{(2\la_1+\la_2+2)^3(3\la_1+\la_2+2)(3\la_1+2\la_2+4)^2}{(2\la_1+\la_2+3)^3(3\la_1+\la_2+4)(3\la_1+2\la_2+5)^2}.
	\end{align*}
\begin{landscape}
	{\bf The case of $i=6$}.
	We have  	
	\begin{align*}
	y_1^{(6)}=&\,x^2-\dfrac{2(2\la_1+\la_2+2)(3\la_1+2\la_2+3)(3\la_1^2+4\la_1\la_2+\la_2^2+8\la_1+5\la_2+6)}{(\la_1+\la_2+2)(2\la_1+\la_2+3)(3\la_1+\la_2+4)(3\la_1+2\la_2+4)}x\\&\quad+\frac{(\la_1+\la_2+1)(2\la_1+\la_2+2)^2(3\la_1+\la_2+3)(3\la_1+2\la_2+3)}{(\la_1+\la_2+2)(2\la_1+\la_2+3)^2(3\la_1+\la_2+4)(3\la_1+2\la_2+5)}.\\
	~\\
	y_2^{(6)}=& \,x^4-\Big(9 \lambda_1^4 (4 \lambda_2+3)+3 \lambda_1^3 \left(30 \lambda_2^2+78 \lambda_2+47\right)+\lambda_1^2 \left(80 \lambda_2^3+372 \lambda_2^2+562 \lambda_2+270\right)\\&+\lambda_1 \left(30 \lambda_2^4+208 \lambda_2^3+528 \lambda_2^2+578 \lambda_2+228\right)+2 (\lambda_2+2)^2 \left(2 \lambda_2^3+11 \lambda_2^2+18 \lambda_2+9\right)\Big)\\&\times\frac{2x^3}{(\lambda_2+1) (\lambda_1+\lambda_2+2) (2 \lambda_1+\lambda_2+3) (3 \lambda_1+\lambda_2+4) (3 \lambda_1+2 \lambda_2+4)}\\&+  \Big(27 \lambda_1^5 (1 + 2 \lambda_2) + 
3 \lambda_1^4 (70 + 147 \lambda_2 + 57 \lambda_2^2) + 
3 \lambda_1^3 (215 + 493 \lambda_2 + 333 \lambda_2^2 + 70 \lambda_2^3) \\&+ 
2 (2 + \lambda_2)^2 (27 + 60 \lambda_2 + 47 \lambda_2^2 + 16 \lambda_2^3 + 2 \lambda_2^4) + 
\lambda_1^2 (978 + 2498 \lambda_2 + 2263 \lambda_2^2 + 880 \lambda_2^3 + 125 \lambda_2^4) \\&+ 
2 \lambda_1 (366 + 1049 \lambda_2 + 1165 \lambda_2^2 + 634 \lambda_2^3 + 170 \lambda_2^4 + 18 \lambda_2^5)\Big)\\&\times\frac{6 (1 + \lambda_1 + \lambda_2) (2 + 2 \lambda_1 + \lambda_2) (3 + 3 \lambda_1 + 2 \lambda_2)x^2}{(\lambda_2+1) (\lambda_1+\lambda_2+2)^2 (2 \lambda_1+\lambda_2+3)^2 (3 \lambda_1+\lambda_2+4) (3 \lambda_1+2 \lambda_2+4)^2 (3 \lambda_1+2 \lambda_2+5)}\\&-\Big(3 \lambda_1^3 (1 + 4 \lambda_2) + 
2 (2 + \lambda_2)^2 (3 + 4 \lambda_2 + \lambda_2^2) + 2 \lambda_1^2 (9 + 28 \lambda_2 + 11 \lambda_2^2) + 
2 \lambda_1 (18 + 47 \lambda_2 + 31 \lambda_2^2 + 6 \lambda_2^3)\Big)\\&\times \frac{2 (1 + \lambda_1 + \lambda_2)^2 (2 + 2 \lambda_1 + \lambda_2)^2 (3 + 3 \lambda_1 + 2 \lambda_2)^2  x}{(1 + \lambda_2) (2 + \lambda_1 + \lambda_2)^3 (3 + 
	2 \lambda_1 + \lambda_2)^3 (4 + 3 \lambda_1 + \lambda_2) (4 + 3 \lambda_1 + 2 \lambda_2) (5 + 3 \lambda_1 + 2 \lambda_2)}\\&+\frac{\lambda_2 (\lambda_1+\lambda_2+1)^3 (2 \lambda_1+\lambda_2+2)^3 (3 \lambda_1+\lambda_2+3) (3 \lambda_1+2 \lambda_2+3)^2}{(\lambda_2+1) (\lambda_1+\lambda_2+2)^3 (2 \lambda_1+\lambda_2+3)^3 (3 \lambda_1+\lambda_2+4) (3 \lambda_1+2 \lambda_2+5)^2}.
\end{align*}
\end{landscape}

\begin{landscape}
	\section{An example of $\g$-stratification of $\mb S\Gr_{11}$}\label{app ex}
	\begin{eg}\label{eg gl-strata}
		The following picture gives an example of the $\g$-stratification of $\mb S\Gr_{11}$. Here we simply write $((\la^{(1)})_{k_1},\dots,(\la^{(n)})_{k_n})$ for $\mb S\Omega_{\bLa,\bm k}$. We also simply write $\la^{(s)}$ for $(\la^{(s)})_{0}$. For instance, $((0,1)_1,(0,1))$ represents $\mb S\Omega_{\bLa,\bm k}$ where $\bLa=(\omega_2,\omega_2)$ and $\bm k=(1,0)$. The arrows represent simple degenerations.
		
		\medskip
		
		\begin{center}
			\begin{tikzpicture}[->,>=stealth',shorten >=1pt,auto,node distance=2.8cm]
			\tikzstyle{every state}=[rectangle,draw=none,text=black]
			
			\node[state]         (S1) at (-11.5, 4.6)        {$((0,1),(0,1),(0,1),(0,1))$};
			\node[state]         (S4) at (-6.5, 4.6)        {$((0,1),(0,1),(0,1),(0,0)_1)$};
			\node[state]         (S2) at (-1.5, 4.6)        {$((0,1),(0,1),(0,0)_1,(0,0)_1)$};
			\node[state]         (S3) at (3.5, 4.6)         {$((0,0)_1,(0,0)_1,(0,0)_1,(0,0)_1)$};
			\node[state]         (xin2) at (-9.25, 0.6)           {$((1,0),(0,1),(0,1))$};
			\node[state]         (xin3) at (-5.75, 0.6)        {$((0,1)_1,(0,1),(0,1))$};
			\node[state]         (xin4) at (-2.25, 0.6)        {$((0,0)_2,(0,1),(0,1))$};
			\node[state]         (xin1) at (-12.75, 0.6)           {$((0,2),(0,1),(0,1))$};
			\node[state]         (xin5) at (1.25, 0.6)           {$((0,1)_1,(0,1),(0,0)_1)$};
			\node[state]         (xin6) at (4.75, 0.6)           {$((0,0)_2,(0,0)_1,(0,0)_1)$};
			\node[state]         (xout2) at (-9.25, -3.4)        {$((1,0),(1,0))$};
			\node[state]         (xout1) at (-12.75, -3.4)        {$((0,2),(0,2))$};
			\node[state]         (xout3) at (-5.75, -3.4)        {$((0,1)_2,(0,1))$};
			\node[state]         (xout4) at (-2.25, -3.4)        {$((0,1)_1,(0,1)_1)$};
			\node[state]         (xout5) at (1.25, -3.4)        {$((0,0)_2,(0,0)_2)$};
			\node[state]         (xout6) at (4.75, -3.4)        {$((0,0)_3,(0,0)_1)$};
			\node[state]         (DC) at (-4, -7.4)           {$((0,0)_4)$};
			
			\path
			(S1) edge (xin1)
			(S1) edge (xin2)
			(S1) edge (xin3)
			(S1) edge (xin4)
			(S2) edge (xin5)
			(S2) edge (xin6)
			(S2) edge (xin4)
			(S3) edge (xin6)
			(S4) edge (xin3)
			(S4) edge (xin5)
			(xin1) edge (xout1)
			(xin1) edge (xout3)
			(xin2) edge (xout2)
			(xin2) edge (xout3)	
			(xin3) edge (xout3)
			(xin3) edge (xout4)	
			(xin4) edge (xout3)
			(xin4) edge (xout5)	
			(xin5) edge (xout3)	
			(xin5) edge (xout4)	
			(xin5) edge (xout6)	
			(xin6) edge (xout5)	
			(xin6) edge (xout6)	
			(xout1) edge (DC)
			(xout2) edge (DC)
			(xout3) edge (DC)
			(xout4) edge (DC)
			(xout5) edge (DC)
			(xout6) edge (DC);		
			\end{tikzpicture}
			
			Example \ref{eg gl-strata}
		\end{center}
	\end{eg}
\end{landscape}	
\end{appendix}
	
\end{document}